\documentclass[10pt]{article}

\usepackage{amsfonts}
\usepackage{amssymb}
\usepackage{amsthm}
\usepackage{amsmath}
\usepackage{bm}
\usepackage{amscd} 
\usepackage{graphicx}
\usepackage{mathrsfs}
\usepackage{amsrefs}
\usepackage{MnSymbol}
\usepackage{hyperref}

\numberwithin{equation}{section} \DeclareMathSizes{2}{10}{12}{13}
\newtheorem{theorem}{Theorem}[section]
\newtheorem{lemma}[theorem]{Lemma}

\newtheorem{examples}[theorem]{Example}
\newtheorem{prop}[theorem]{Proposition}
\newtheorem{corollary}[theorem]{Corollary}
\newtheorem{remark}[theorem]{Remark}
\newtheorem{definition}[theorem]{Definition}

\numberwithin{equation}{section}

\title{Cohomology of modules over $H$-categories and co-$H$-categories}
\author{Mamta Balodi\footnote{mamta.balodi@gmail.com} \footnote{MB was supported by  SERB Fellowship PDF/2017/000229} $\qquad$ Abhishek Banerjee\footnote{abhishekbanerjee1313@gmail.com} \footnote{AB was partially supported by SERB Matrics fellowship MTR/2017/000112} $\qquad$ Samarpita Ray\footnote{ray.samarpita31@gmail.com} }
\date{}

\begin{document}
\maketitle
\centerline{\emph{Department of Mathematics, Indian Institute of Science, Bangalore 560012, India.}}

\begin{abstract} Let $H$ be a Hopf algebra. We consider $H$-equivariant modules over a Hopf module category $\mathcal C$
as modules over the smash extension $\mathcal C\# H$. We construct Grothendieck spectral sequences for the cohomologies
as well as the $H$-locally finite cohomologies of these objects. We also introduce relative $(\mathcal D,H)$-Hopf modules
over a Hopf comodule category $\mathcal D$. These generalize relative $(A,H)$-Hopf modules over an $H$-comodule algebra
$A$. We construct Grothendieck spectral sequences for their cohomologies by using their rational  $Hom$ objects and
higher derived functors of coinvariants. 
\end{abstract}

\emph{MSC(2010) Subject Classification: 16S40, 16T05, 18E05}

\smallskip


\section{Introduction}

\smallskip
Let $H$ be a Hopf algebra over a field $K$. An $H$-category is a small $K$-linear category $\mathcal C$ such that the morphism
space $Hom_{\mathcal C}(X,Y)$ is an $H$-module for each couple of objects $X$, $Y\in Ob(\mathcal C)$ and the   composition
of morphisms in $\mathcal C$ is well-behaved with respect to the action of  $H$. Similarly, a co-$H$-category is a small
$K$-linear category $\mathcal D$ such that  the morphism
space $Hom_{\mathcal D}(X,Y)$ is an $H$-comodule for each couple of objects $X$, $Y\in Ob(\mathcal D)$ and the composition
of morphisms in $\mathcal D$ is well-behaved with respect to the coaction of $H$. In other words, an $H$-category is enriched over the monoidal category of $H$-modules and a
 co-$H$-category is enriched over the monoidal category of $H$-comodules. The purpose of this paper is to study cohomology in module categories over $H$-categories and co-$H$-categories. 

\smallskip
The Hopf module categories that we use were first considered by Cibils and Solotar \cite{CiSo}, where they discovered a Morita equivalence that relates Galois coverings of a category to its smash extensions via a Hopf algebra. We view these $H$-categories
and the modules over them as objects of interest in their own right. We recall here that an ordinary ring may be expressed as a preadditive category with a single object. Accordingly, an arbitrary small preadditive category may be understood as a `ring with several objects' (see Mitchell \cite{Mit1}). As such, the theories obtained by replacing rings  by preadditive categories have been developed widely in the literature (see, for instance, \cite{AB1}, \cite{EV}, \cite{LoVa}, \cite{LoVa1},\cite{Lo},\cite{Xu1}, \cite{Xu2}). In this respect, an $H$-category may be seen as an ``$H$-module algebra with several objects''. Likewise, a co-$H$-category may be seen as an ``$H$-comodule algebra with several objects.''

\smallskip
The various aspects of categorified Hopf  actions and coactions on algebras have already been studied by several authors. In \cite{HS07}, Herscovich and Solotar obtained a Grothendieck spectral sequence for the Hochschild-Mitchell cohomology of
an $H$-comodule category appearing as an $H$-Galois extension. Hopf comodule categories were also studied in \cite{StSt}, where the authors introduced cleft $H$-comodule categories and extended classical results on cleft comodule algebras. More recently,
Batista, Caenepeel and Vercruysse have shown in \cite{BCV} that several deep theorems on Hopf modules can be extended to a categorification of Hopf algebras (see also \cite{CF}). 

\smallskip
In this paper, we will construct a Grothendieck spectral sequence that computes the higher derived $Hom$ functors for $H$-equivariant
modules over an $H$-category $\mathcal C$. We will also construct a spectral sequence that gives the higher derived $Hom$ functors
for relative $(\mathcal D,H)$-modules, where $\mathcal D$ is a co-$H$-category. We will develop these cohomology theories
in a manner analogous to the ``$H$-finite cohomology'' obtained by Gu\'{e}d\'{e}non \cite{Gue} (see also \cite{Gue1}) and the cohomology of relative
Hopf modules studied by Caenepeel and Gu\'{e}d\'{e}non in \cite{CanGue} respectively.

\smallskip
We now describe the paper in more detail. We begin in Section 2 by recalling the notion of a left $H$-category and a right
co-$H$-category. For a left $H$-category $\mathcal C$, we have a category of $H$-invariants which will be  denoted by $\mathcal C^H$. 
For a right co-$H$-category $\mathcal D$, there is a corresponding category of $H$-coinvariants which will be denoted by
$\mathcal D^{coH}$.  If $H$ is a finite dimensional Hopf algebra and $H^*$ is its linear dual, then a $K$-linear category
$\mathcal D$ is a left $H^*$-category if and only if it is a right co-$H$-category. In that case, $\mathcal D^{H^*}=\mathcal D^{coH}$.

\smallskip
In Sections 3 and 4, we work with a left $H$-category $\mathcal C$. We consider right $\mathcal C$-modules that are equipped
with an additional left $H$-equivariant structure (see Definition \ref{Dx3.2}). This category is denoted by $(Mod\text{-}\mathcal C)^H_H$. If $\mathcal M$,
$\mathcal N\in (Mod\text{-}\mathcal C)^H_H$, the space $Hom_{Mod\text{-}\mathcal C}(\mathcal M,\mathcal N)$
of right $\mathcal C$-module morphisms carries a left $H$-module structure whose $H$-invariants are given by
$Hom_{Mod\text{-}\mathcal C}(\mathcal M,\mathcal N)^H=Hom_{(Mod\text{-}\mathcal C)^H_H}(\mathcal M,\mathcal N)$.

\smallskip
More precisely, let $(Mod\text{-}\mathcal C)_H$ denote the category with the same objects as $(Mod\text{-}\mathcal C)^H_H$ but whose morphisms are ordinary $\mathcal C$-modules morphisms. Then, we show that $(Mod\text{-}\mathcal C)_H$ is a left $H$-category and $(Mod\text{-}\mathcal C)^H_H$ may be recovered as the category of $H$-invariants of $(Mod\text{-}\mathcal C)_H$.
Further, we obtain that $(Mod\text{-}\mathcal C)^H_H$ is identical to the category $Mod\text{-}(\mathcal C\# H)$ of 
right modules over the smash product category $\mathcal C\# H$. In particular, this shows that $(Mod\text{-}\mathcal C)^H_H$
is a Grothendieck category. We then construct a Grothendieck spectral sequence (see Theorem \ref{Tc3.15})
\begin{equation*}
R^{p}(-)^H\left(\textnormal{Ext}^q_{Mod\text{-}\mathcal{C}}(\mathcal{M},\mathcal N)\right)\Rightarrow \left(R^{p+q}Hom_{Mod\text{-}(\mathcal{C}\#H)}(\mathcal{M},{-})\right)(\mathcal{N})
\end{equation*}
for the higher derived $Hom$ in $Mod\text{-}(\mathcal C\# H)$ in terms of the derived $Hom$ in $Mod\text{-}\mathcal C$
and the derived functor of $H$-invariants. 

\smallskip
We proceed in Section 4 to develop the ``$H$-finite cohomology'' of $(\mathcal C\# H)$-modules in a manner analogous
to Gu\'{e}d\'{e}non \cite{Gue}. If $M$ is an $H$-module, we denote by $M^{(H)}$ the collection of all elements $m\in M$
such that $Hm$ is a finite dimensional vector space. In particular, $M$ is said to be $H$-locally finite if $M^{(H)}=M$ and we let
$H$-$mod$ denote the category of $H$-locally finite modules. This leads
to a functor
\begin{equation*}
\mathcal L_{Mod\text{-}\mathcal C}:(Mod\text{-}(\mathcal C\# H))^{op}\times Mod\text{-}(\mathcal C\# H)
\longrightarrow H\text{-}mod\qquad \mathcal L_{Mod\text{-}\mathcal C}(\mathcal M,\mathcal N):=Hom_{Mod\text{-}\mathcal C}(\mathcal M,\mathcal N)^{(H)}
\end{equation*}  We then construct a Grothendieck spectral sequence (see Theorem \ref{Tb4.19})
\begin{equation*}
R^p{(-)}^{(H)}({\textnormal{Ext}}^q_{Mod\text{-}\mathcal{C}}(\mathcal{M},\mathcal{N}))\Rightarrow \left(\mathbf{R}^{p+q}\mathcal{L}_{Mod\text{-}\mathcal{C}}(\mathcal{M},-)\right)(\mathcal{N})
\end{equation*}
The left $H$-category $\mathcal C$ is said to be locally finite if every morphism space $Hom_{\mathcal C}(X,Y)$ is locally finite 
as an $H$-module. We denote by $mod\text{-}(\mathcal C\# H)$ the full subcategory of $Mod\text{-}(\mathcal C\# H)$ consisting
of those left $H$-equivariant right $\mathcal C$-modules $\mathcal M$ such that $\mathcal M(X)$ is $H$-locally finite for 
each $X\in Ob(\mathcal C)$. When $\mathcal C$ is left $H$-locally finite and right noetherian, we construct a spectral sequence 
(see Theorem \ref{Tb4.18})
\begin{equation*}
R^{p}(-)^H\big(\textnormal{Ext}^q_{Mod\text{-}\mathcal{C}}(\mathcal{M},\mathcal{N})\big)\Rightarrow \left(R^{p+q}Hom_{mod\text{-}(\mathcal{C}\#H)}(\mathcal{M},-)\right)(\mathcal{N})
\end{equation*}
In Section 5, we work with a right co-$H$-category $\mathcal D$ and introduce the category ${_{\mathcal{D}}}{\mathscr{M}}^H$ of relative $(\mathcal D,H)$-Hopf modules  (see Definition \ref{Dr5.1}).
A relative $(\mathcal D,H)$-module consists of an $H$-coaction on a pair $(\mathcal D,\mathcal M)$, where $\mathcal M$
is a left $\mathcal D$-module. In particular, $\mathcal M(X)$ is equipped with the structure of a right $H$-comodule for each
$X\in Ob(\mathcal D)$. We show that ${_{\mathcal{D}}}{\mathscr{M}}^H$ is a Grothendieck category.

\smallskip
Let $Comod\text{-}H$ be the category of $H$-comodules. Thereafter, we construct a functor (see \eqref{HOM})
\begin{equation*}
HOM_{\mathcal D\text{-}Mod}:({_{\mathcal{D}}}{\mathscr{M}}^H)^{op}
\times {_{\mathcal{D}}}{\mathscr{M}}^H\longrightarrow Comod\text{-}H
\end{equation*} by using the right adjoint of the functor $\mathcal N\otimes (-): Comod\text{-}H\longrightarrow {_{\mathcal{D}}}{\mathscr{M}}^H$ for each fixed $\mathcal N\in {_{\mathcal{D}}}{\mathscr{M}}^H$. In the case of an $H$-comodule algebra as considered by Caenepeel and Gu\'{e}d\'{e}non, 
the $HOM$ functor gives the collection of ``rational morphisms'' between relative Hopf modules (see\cite[$\S$ 2]{CanGue}). Although the category ${_{\mathcal{D}}}{\mathscr{M}}^H$ is not necessarily enriched over $Comod\text{-}H$, we see that 
$HOM_{\mathcal D\text{-}Mod}(\mathcal M,\mathcal N)$ behaves like a $Hom$ object. The morphisms in 
$Hom_{{_{\mathcal{D}}}{\mathscr{M}}^H}(\mathcal M,\mathcal N)$ may be recovered as the $H$-coinvariants
$HOM_{\mathcal D\text{-}Mod}(\mathcal M,\mathcal N)^{coH}=Hom_{{_{\mathcal{D}}}{\mathscr{M}}^H}(\mathcal M,\mathcal N)$. We then construct a Grothendieck spectral sequence (see Theorem \ref{Tf5.9})
\begin{equation*}
R^{p}(-)^{coH}(R^qHOM_{\mathcal{D}\text{-}Mod}(\mathcal{M},-)(\mathcal{N}))\Rightarrow \left(R^{p+q}Hom_{{_{\mathcal{D}}}{\mathscr{M}}^H}(\mathcal{M},-)\right)(\mathcal{N})
\end{equation*} For $\mathcal M$, $\mathcal N\in {_{\mathcal{D}}}{\mathscr{M}}^H$ with $\mathcal M$ finitely generated
as a $\mathcal D$-module, we show that $Hom_{\mathcal D\text{-}Mod}(\mathcal M,\mathcal N)$ is an $H$-comodule and that $HOM_{\mathcal D\text{-}Mod}(\mathcal M,\mathcal N)=Hom_{\mathcal D\text{-}Mod}(\mathcal M,\mathcal N)$. When 
$\mathcal D$ is also left noetherian, we construct a Grothendieck spectral sequence (see Theorem \ref{Tf5.17})
\begin{equation*}
R^{p}(-)^{coH}\big(\textnormal{Ext}^q_{\mathcal{D}\text{-}Mod}(\mathcal{M},\mathcal{N})\big)\Rightarrow \left({R}^{p+q}Hom_{{_{\mathcal{D}}}{\mathscr{M}}^H}(\mathcal{M},-)\right)(\mathcal{N})
\end{equation*} for the higher derived $Hom$ in ${_{\mathcal{D}}}{\mathscr{M}}^H$.

\smallskip
\textbf{Notations:} Throughout the paper,  $K$ is a field, $H$ is a Hopf algebra with comultiplication $\Delta$, counit $\varepsilon$ and  bijective antipode $S$. We shall use  Sweedler's notation for the coproduct $\Delta(h)=\sum h_1 \otimes h_2$ and for a coaction $\rho:M \longrightarrow M \otimes H$, $\rho(m)=\sum m_0 \otimes m_1$. We denote by $H^*$ the linear dual of $H$. The category of left $H$-modules will be denoted by $H$-$Mod$ and the category of right $H$-comodules will be denoted by $Comod$-$H$. For $M \in H$-$Mod$, we set  $M^H:=\{m \in M ~|~ hm=\varepsilon(h)m~ \forall h \in H\}$. For $M \in Comod$-$H$, we set  $M^{coH}:=\{m\in M~|~ \rho(m)= m \otimes 1_H\}$.

\section{$H$-categories and co-$H$-categories}\label{basic}

\smallskip
Let $H$ be a Hopf algebra over a field $K$. Then, it is well known (see, for instance, \cite[$\S$ 2.2]{Psch1}) that the category of $H$-modules as well as the category
of $H$-comodules is monoidal. A  $K$-linear category is said to be an $H$-module category (resp. an $H$-comodule category)  if it is enriched over the monoidal category of 
$H$-modules (resp. $H$-comodules). For more on enriched categories, the reader may see, for example, \cite[Chapter 6]{Borc} or \cite{GMK}. 

\begin{definition}\label{defH-cat} (see Cibils and Solotar \cite[Definition 2.1]{CiSo})
Let  $K$ be a field. A  $K$-linear category $\mathcal C$ is said to be a left $H$-module category  if it is enriched over the monoidal category of 
left $H$-modules. In other words, it satisfies the following
conditions:
\begin{itemize}
\item[(i)] $Hom_\mathcal{C}(X,Y)$ is a left $H$-module for all $X,Y \in Ob(\mathcal{C})$. 
\item[(ii)] $h(\text{id}_X)=\varepsilon(h)\cdot \text{id}_X$ for every $X \in Ob(\mathcal{C})$ and every $h \in H$. 
\item[(iii)] The composition of morphisms in $\mathcal C$ is $H$-equivariant, i.e., for any $h \in H$ and any pair of composable morphisms $g:X \longrightarrow Y$, $f: Y \longrightarrow Z$, we have
\begin{equation*}h(fg)=\sum h_1(f)h_2(g)
\end{equation*}
\end{itemize}
By a left $H$-category, we will always mean a small left $H$-module category. A right $H$-category may be
defined similarly. 
\end{definition}

\begin{definition}\label{Dx2.2}
Let $\mathcal{C}$ be a left $H$-module category. A morphism $f\in Hom_{\mathcal C}(X,Y)$ is said to be $H$-invariant if $h(f)=\varepsilon(h)\cdot f$ for all $h \in H$. The subcategory whose objects are the same as those of $\mathcal C$
and whose morphisms are the $H$-invariant morphisms in $\mathcal C$ is denoted by
$\mathcal C^H$. 
\end{definition}

Let $A$ be a left $H$-module algebra. A right $A$-module $M$ is said to be left $H$-equivariant  if
\begin{itemize}
\item[(i)] $M$ is a left $H$-module and
\item[(ii)] the action of $A$ on $M$ is a morphism of $H$-modules, i.e., $h(ma)=\sum h_1(m)h_2(a)$, for all $h \in H,~ a \in A$ and $m \in M.$
\end{itemize}

\begin{examples}\label{egH} (see \cite{kk}) Let $A$ be a left $H$-module algebra.
\begin{itemize}
\item[(i)]  Then, the category $_H{M}{^A}$  of (isomorphism classes of) all left $H$-equivariant finitely generated right $A$-modules, with  right $A$-module morphisms between them, is an $H$-category. In fact, one can check that for $X,Y \in Ob(_H{M}{^A})$, the morphism space $Hom_A(X,Y)$ is a left $H$-module via
\begin{equation*}h(f)(x)=\sum h_1f(S(h_2)x) \qquad \forall \textrm{ }x\in X, \textrm{ }\forall\textrm{ }f\in Hom_A(X,Y)
\end{equation*}  

\item[(ii)] The finitely generated free right $A$-modules are automatically
left $H$-equivariant. The category  of (isomorphism classes of) finitely generated free right $A$-modules is an $H$-category.
\end{itemize}
\end{examples}

We may also define the notion of a co-$H$-category, which replaces an $H$-comodule algebra (see \cite{StSt}). This notion also appears implicitly
in \cite{CiSo}. 
\begin{definition}\label{coHcat}  
By a right co-$H$-category, we will mean a small  $K$-linear category $\mathcal{D}$  that  is enriched over the monoidal category of 
right $H$-comodules. In other words, we have:
\begin{enumerate}
\item[(i)] $Hom_\mathcal{D}(X,Y)$ is a right $H$-comodule for all $X,Y \in Ob(\mathcal{D}),$ with structure map
$$\rho_{XY}:Hom_\mathcal{D}(X,Y) \longrightarrow Hom_\mathcal{D}(X,Y) \otimes H, \quad \rho_{XY}(f)=\sum f_0 \otimes f_1$$ 
\item[(ii)] $\rho_{XX}(\text{id}_X)=\text{id}_X \otimes 1_H$, for any $X \in Ob(\mathcal{D})$ and any $h \in H$.
\item[(iii)] The composition of morphisms in $\mathcal D$ is $H$-coequivariant, i.e., for any pair of composable morphisms $g:X \longrightarrow Y$, $f: Y \longrightarrow Z$,  we have
\begin{equation*}\rho_{XZ}(fg)=\sum {(fg)}_0 \otimes {(fg)}_1 =\sum f_0g_0 \otimes f_1g_1=\rho_{YZ}(f)\rho_{XY}(g)
\end{equation*}
\end{enumerate} A left co-$H$-category may be defined similarly.

\smallskip  A morphism $f\in Hom_{\mathcal D}(X,Y)$ in a right co-$H$-category  is said to be $H$-coinvariant if it satisfies $\rho_{XY}(f)=f\otimes 1_H$. The subcategory whose objects are the same as those of $\mathcal D$ and whose morphisms are $H$-coinvariant is denoted
by $\mathcal D^{coH}$. 
\end{definition}

\begin{prop}\label{P2.5x}
Let $H$ be a finite dimensional Hopf algebra and let $\mathcal{D}$ be a small  $K$-linear category. Then,  $\mathcal{D}$ is a right co-$H$-category if and only if $\mathcal{D}$ is a left $H^*$-category. Moreover, $\mathcal{D}^{H^*} = \mathcal{D}^{coH}$.
\end{prop}

\begin{proof}
Let $\{e_1,\dots,e_n\}$ be a basis of $H$ and let $\{e_1^*,\dots,e_n^*\}$ be its dual basis. If $\mathcal{D}$ is a right co-$H$-category, then $\mathcal{D}$ becomes a left $H^*$-category with
\begin{equation*}h^*(f) :=\sum f_0 h^*(f_1)
\end{equation*} for all $h^* \in H^*$ and $f \in Hom_{\mathcal D}(X,Y)$.
Indeed, it is easy to check that this action makes $Hom_\mathcal{D}(X,Y)$  a left $H^*$-module for every $X,Y \in Ob(\mathcal{D})$ and  that
\begin{align*}
h^*(fg) &= \sum(fg)_0h^*((fg)_1)= \sum f_0g_0 h^*(f_1g_1)
=  \sum f_0g_0 h^*_1(f_1)h^*_2(g_1)=  \sum f_0 h^*_1(f_1)g_0h^*_2(g_1)= \sum h^*_1(f)h^*_2(g).
\end{align*}
Conversely, if $\mathcal{D}$ is a left $H^*$-category, then $\mathcal{D}$ is a right co-$H$-category with 
\begin{equation*}\rho_{XY}:Hom_\mathcal{D}(X,Y) \longrightarrow Hom_\mathcal{D}(X,Y) \otimes H, \quad \rho_{XY}(f):=\sum_{i=1}^n e_i^*(f)\otimes e_i
\end{equation*}
It may be verified that this gives a right $H$-comodule structure on $Hom_\mathcal{D}(X,Y)$. We need to check that the composition of morphisms in $\mathcal{D}$ is $H$-coequivariant.  For any $h^* \in H^*$, $ g \in Hom_\mathcal{D}(X,Y)$ and $f \in Hom_\mathcal{D}(Y,Z)$, we have
\begin{equation*}
\begin{array}{ll}
(id\otimes h^*)(\rho_{XZ}(fg))=(id\otimes h^*)\left(\sum_{i=1}^n e_i^*(fg) \otimes e_i\right) &=\sum_{i=1}^n e_i^*(fg) \otimes h^*(e_i)\\
&=\sum_{i=1}^n(h^*(e_i)e_i^*)(fg) \otimes 1_H\\
&=h^*(fg) \otimes 1_H\\
&=\sum_{i=1}^nh_1^*(f)h_2^*(g)\otimes 1_H\\
&=\sum_{1\leq i,j\leq n}(h^*_1(e_i)e_i^*) (f)(h^*_2(e_j)e_j^*) (g)\otimes 1_H\\
&= \sum_{1\leq i,j\leq n}e_i^* (f)e_j^*(g)\otimes h^*(e_ie_j)\\
&= (id\otimes h^*)\left(\sum_{1\leq i,j\leq n}e_i^* (f)e_j^*(g)\otimes e_ie_j\right)
\end{array}
\end{equation*}
Since $H$ is finite dimensional, it follows that 
\begin{equation*}\rho_{XZ}(fg)=\sum_{1\leq i,j\leq n}e_i^*(f)e_j^*(g)\otimes e_ie_j=\rho_{YZ}(f)\rho_{XY}(g)
\end{equation*} 
We also have
\begin{align*}
Hom_{\mathcal{D}^{H^*}}(X,Y) &= \{f \in Hom_{\mathcal D}(X,Y)\  |\ h^*(f)= \varepsilon_{H^*} (h^*)f=h^*(1_H)f,\textrm{ } \forall
 h^* \in H^*\}\\
&= \{f \in Hom_{\mathcal D}(X,Y)\  |\  \sum f_0h^*(f_1)= h^*(1_H)f, \textrm{ }\forall h^* \in H^*\}\\
&=\{f \in Hom_{\mathcal D}(X,Y)\  |\  (id \otimes h^*)(\rho_{XY}(f)) = (id \otimes h^*)(f\otimes 1_H), ~\forall h^*\in H^*\}\\
&=\{f \in Hom_{\mathcal D}(X,Y)\  |\  \rho_{XY}(f)= f\otimes 1_H\} =Hom_{ \mathcal{D}^{coH}}(X,Y). 
\end{align*} 
\end{proof}

\begin{remark}
\emph{Using   Example \ref{egH} and Proposition \ref{P2.5x}, we can obtain several examples of co-$H$-categories. Another example
of a co-$H$-category is the smash extension $\mathcal C\# H$, which will be recalled in the next section.}
\end{remark}

\section{$H$-equivariant modules and the first spectral sequence}\label{ModHcat}

\smallskip
Let $\mathcal C$ be a left $H$-category. In this section, we will study the category of $H$-equivariant $\mathcal C$-modules
and compute their higher derived $Hom$ functors by means of a spectral sequence.   We begin with the following definition  (see, for instance, \cite{ Sten,Mit2 }). 

\begin{definition}
Let $\mathcal{C}$ be a small  $K$-linear category. A right module  over $\mathcal{C}$ is  a  $K$-linear functor $\mathcal{C}^{op} \longrightarrow Vect_K$, where $Vect_K$ denotes the category of  $K$-vector spaces. Similarly, a left module over $\mathcal{C}$ is a  $K$-linear functor $\mathcal{C} \longrightarrow Vect_K$. The category of all right (resp. left) modules over $\mathcal{C}$ will be denoted by $Mod$-$\mathcal C$  (resp. $\mathcal C$-$Mod$). 
\end{definition}

For each  $X\in Ob(\mathcal C)$, the representable functors ${\bf h}_X:=Hom_{\mathcal C}(-,X)$ and $_X{\bf h}:=Hom_{\mathcal C}(X,-)$  are examples of right and left modules over $\mathcal{C}$ respectively. Unless otherwise mentioned,
by a $\mathcal C$-module we will always mean a right  $\mathcal C$-module. 

\begin{definition}\label{Dx3.2}
Let  $\mathcal{C}$ be a left $H$-category. Let $\mathcal{M}$ be a right $\mathcal C$-module with a given left $H$-module structure on 
$\mathcal M(X)$ for each $X\in Ob(\mathcal C)$. Then,  $\mathcal{M}$ is said to be a left $H$-equivariant right $\mathcal{C}$-module if 
\begin{equation*}
h(\mathcal{M}(f)(m))=\sum \mathcal{M}(h_2f)(h_1m) \quad\forall ~ h \in H,~ f\in Hom_{\mathcal C}(X,Y) , ~m \in \mathcal{M}(Y)
\end{equation*} A morphism $\eta:\mathcal M\longrightarrow\mathcal N$ of left $H$-equivariant right $\mathcal C$-modules is a morphism
$\eta\in Hom_{Mod\text{-}\mathcal C}(\mathcal M,\mathcal N)$ such that 
$\eta(X):\mathcal M(X)\longrightarrow\mathcal N(X)$ is $H$-linear for each $X\in Ob(\mathcal C)$. We will denote the category of
left $H$-equivariant right $\mathcal C$-modules by $(Mod\text{-}\mathcal C)_H^H$. 

\smallskip
By $(Mod\text{-}\mathcal C)_H$, we will denote the category whose objects are the same as those of  $(Mod\text{-}\mathcal C)_H^H$,  but whose
morphisms are those of right $\mathcal C$-modules. 
\end{definition}

\begin{lemma}  Let  $\mathcal{C}$ be a left $H$-category. 
 Given   $\mathcal M,\mathcal N\in (Mod\text{-}\mathcal C)_H$, the $H$-module action on $Hom_{Mod\text{-}\mathcal C}(\mathcal M,\mathcal N)$ given by
\begin{equation}\label{3xee3/1}
(h\cdot \eta)(X)(m)=\sum h_1\eta (X)(S(h_2)m)
\end{equation}
for $\eta\in Hom_{Mod\text{-}\mathcal C}(\mathcal M,\mathcal N)$, $h\in H$, $X\in Ob(\mathcal C)$, $m\in \mathcal M(X)$ makes $ (Mod\text{-}\mathcal C)_H$ a left $H$-category. 
\end{lemma}

\begin{proof} 
Using the $H$-equivariance of $\mathcal M$ and $\mathcal N$,  it may be verified that the action in \eqref{3xee3/1} defines a left $H$-module structure on $Hom_{Mod\text{-}\mathcal C}(
\mathcal M,\mathcal N)$. We now  consider $\eta \in Hom_{Mod\text{-}\mathcal C}(\mathcal M,\mathcal N)$ and $\nu\in Hom_{Mod\text{-}\mathcal C}(\mathcal N,\mathcal P)$. Then, we have
\begin{equation*}
\begin{array}{ll}
\sum ((h_1\nu)(X)\circ (h_2\eta)(X))(m) & = \sum (h_1\nu) (X)\left( h_2\eta (X)(S(h_3)m)\right)\\
&= \sum h_1\nu(X)\left(S(h_2) h_3\eta (X)(S(h_4)m)\right)\\
&= \sum h_1\nu(X)\left(\eta (X)(\varepsilon(h_2)S(h_3)m)\right) \\
&= \sum h_1\nu(X)\left(\eta (X)(S(h_2)m)\right)  \\
&=( h\cdot (\nu\circ \eta))(X)(m)
\end{array}
\end{equation*} This proves the result. 
\end{proof}

\begin{prop}\label{Pxp3.4}
 The category $(Mod\text{-}\mathcal C)_H^H$ of left $H$-equivariant right $\mathcal C$-modules is identical
to $((Mod\text{-}\mathcal C)_H)^H$. 
\end{prop}

\begin{proof}
Suppose that $\eta\in Hom_{Mod\text{-}\mathcal C}(
\mathcal M,\mathcal N)^H$. We claim that $\eta(X):\mathcal M(X)\longrightarrow \mathcal N(X)$ is $H$-linear for each
$X\in Ob(\mathcal C)$. For this, we observe that
\begin{equation*}
\begin{array}{ll}
\eta(X)(hm)  = \sum \eta(X)(\varepsilon(h_1)h_2m) = \sum \varepsilon(h_1)\eta(X)(h_2m)
& = \sum (h_1\cdot \eta)(X)(h_2m)\\
&= \sum h_1\eta(X)(S(h_2)h_3m)
 =h\eta(X)(m)\\
\end{array}
\end{equation*} for any $h\in H$ and $m\in \mathcal M(X)$. Conversely, if each  $\eta(X):\mathcal M(X)\longrightarrow \mathcal N(X)$ is $H$-linear, it is clear from the definition of the left $H$-action in \eqref{3xee3/1}  that $h\cdot \eta=\varepsilon(h)\eta$, i.e., $\eta\in Hom_{Mod\text{-}\mathcal C}(
\mathcal M,\mathcal N)^H$.
\end{proof}

We will now study the category $(Mod\text{-}\mathcal C)_H^H$ of left $H$-equivariant right $\mathcal C$-modules. In particular,
one may ask if  $(Mod\text{-}\mathcal C)_H^H$ is an abelian category. We will show that  $(Mod\text{-}\mathcal C)_H^H$ is in fact a Grothendieck category. For this, we will need to consider the smash product category of $\mathcal C$ and $H$. 

\begin{definition} (see \cite[$\S$ 2]{CiSo})
Let $\mathcal{C}$ be a left $H$-category. The smash product of $\mathcal{C}$
and $H$, denoted by $\mathcal{C}\# H$,  is the  $K$-linear category defined by
\begin{equation*}
\begin{array}{c}
Ob(\mathcal{C}\# H) := Ob(\mathcal{C})\qquad  Hom_{\mathcal{C}\# H}(X,Y) := Hom_{\mathcal{C}}(X,Y)\otimes H\\
\end{array}
\end{equation*}
 An element of $Hom_{\mathcal{C}\# H}(X,Y) $ is a finite sum of the form $\sum g_i\# h_i$, with $g_i\in Hom_{\mathcal C}(X,Y)$
 and $h_i\in H$. Then, the composition of morphisms in $\mathcal C\# H$  is determined by
\begin{equation*}(f\# h) (g \# h') = \sum f(h_1g)\#(h_2h')
\end{equation*}
for any pair of composable morphisms $g:X \longrightarrow Y$, $f: Y \longrightarrow Z$ in $\mathcal{C}$ and any $h,h' \in H$. 
\end{definition}

\begin{lemma}\label{2.2}
Let $\mathcal{M}\in Mod$-$(\mathcal C\# H)$. Then, $\mathcal{M}(X)$ has a left $H$-module structure for each  $X\in Ob(\mathcal{C}\#H)$  given by 
\begin{equation*} hm := \mathcal{M}(id_X \# S(h))( m)\qquad \forall ~ h\in H, m\in \mathcal M(X)
\end{equation*} Further, given any morphism $\eta:\mathcal{M}\longrightarrow \mathcal{N}$ in $Mod$-$(\mathcal C\# H)$,  every $\eta(X):\mathcal{M}(X)\longrightarrow \mathcal{N}(X)$ is $H$-linear.
\end{lemma}
\begin{proof}
For $h,h' \in H$ and $m \in \mathcal{M}(X)$, we have
\begin{align*}
h(h' m) &= (\mathcal{M}(id_X \# S(h))\circ \mathcal{M}(id_X \# S(h')))(m)\\
					 &= \mathcal{M}((id_X \# S(h'))(id_X \# S(h)))(m)\\
					 &= \sum \mathcal{M}(S(h_2')(id_X) \# S(h_1')S(h))(m)\\
					 &= \sum\mathcal{M}(\varepsilon(S(h_2'))(id_X) \# S(h_1')S(h))(m) \quad  \text{(using 							Definition \ref{defH-cat}(ii))}\\
					 &= \mathcal{M}(id_X \# S(hh'))(m)\quad \text{(using 							$\varepsilon \circ S=\varepsilon$)}\\ &=(hh')(m)			 
\end{align*}
If $\eta:\mathcal{M}\longrightarrow \mathcal{N}$ is a morphism in $Mod$-$(\mathcal C\# H)$,  it may be verified easily that each $\eta(X):\mathcal{M}(X)\longrightarrow \mathcal{N}(X)$ is $H$-linear.
\end{proof}

\begin{prop}\label{2.3} Let $\mathcal C$ be a left $H$-category. 
Then, there is a one-one correspondence between left  $H$-equivariant right $\mathcal C$-modules and right  modules over
$\mathcal C\# H$.  
\end{prop}
\begin{proof}
For any $H$-equivariant $\mathcal C$-module $\mathcal{M}$, we have the object $\mathcal{M}'$ in $Mod\text{-}(\mathcal{C}\# H)$  defined by
\begin{equation}\label{weq2.1}
\begin{array}{c}
\mathcal{M}'(X) := \mathcal{M}(X) \ \ \ \forall \ X \in Ob(\mathcal{C}\# H), \\
\mathcal{M}'(f \# h)(m):=S^{-1}(h) \mathcal{M}(f)(m) \ \ \ \forall \ f \# h \in Hom_{\mathcal{C}\#H}(Y,X),~ m \in \mathcal{M}(X).
\end{array}
\end{equation}
For $f' \# h' \in Hom_{\mathcal{C}\#H}(Z,Y)$,  $ f \# h \in Hom_{\mathcal{C}\#H}(Y,X)$ and $m \in \mathcal{M}(X)$, we have
\begin{equation*}
\begin{array}{ll}
\mathcal{M}' \big((f\#h)\circ (f'\#h')\big)(m) &=\sum \mathcal{M}' \big(f(h_1f')\#h_2h'\big)(m)\\
&=\sum S^{-1}(h_2h')\mathcal{M}\big(f(h_1f')\big)(m)\\
&=S^{-1}(h')\sum S^{-1}(h_2)\mathcal{M}\big(f(h_1f')\big)(m)\\
&=S^{-1}(h')\sum \mathcal{M}\big(S^{-1}(h_2)(f(h_1f'))\big)(S^{-1}(h_3)m)\ \ \ \ \ \ (\textnormal{since\ } \mathcal{M}\text{ is $H$-equivariant} )\\
&=S^{-1}(h')\sum \mathcal{M}\big(S^{-1}(h_3)(f)S^{-1}(h_2)(h_1f'))\big)(S^{-1}(h_4)m)\\
&=S^{-1}(h')\sum \mathcal{M}\big((S^{-1}(h_1)(f))\circ f')\big)(S^{-1}(h_2)m)\\
&=S^{-1}(h')\mathcal M(f')\left(\sum \mathcal{M}(S^{-1}(h_1)f)(S^{-1}(h_2)m)\right)\\
&=S^{-1}(h')\mathcal M(f')\left(S^{-1}(h)\mathcal M(f)(m))\right)\ \ \ \ \ \ \ \ \ \ \ \ \ \ \ \ \ \ \ (\textnormal{since\ } \mathcal{M}\text{ is $H$-equivariant} )\\
&= (\mathcal M'(f'\# h')\circ \mathcal M'(f\# h))(m)
\end{array}
\end{equation*}
Conversely, given any $\mathcal{M}'$ in $Mod\text{-}(\mathcal{C}\# H)$,  we can obtain an $H$-equivariant $\mathcal C$-module   defined by
\begin{equation*}
\begin{array}{c}\mathcal{M}(X) := \mathcal{M}'(X)\ \ \ \forall \ X \in Ob(\mathcal{C})\\
\mathcal{M}(f) :=  \mathcal{M}'(f \# 1_H) \ \ \ \forall\ f \in Hom_{\mathcal{C}}(Y,X)
\end{array}
\end{equation*} From Lemma \ref{2.2}, it follows that $\mathcal M(X)=\mathcal M'(X)$ has a left $H$-module structure. 
We now check that $\mathcal{M}$ is indeed $H$-equivariant:
\begin{align*}
h (\mathcal{M}(f)(m)) &= \mathcal{M}'(\text{id}_X \# S(h))(\mathcal{M}'(f \# 1_H)(m))\\
					&=  \mathcal{M}'(f \# S(h))(m)\\
					&= \sum \mathcal{M}'((id_X \# S(h_1))(h_2f \# 1_H))(m)\\
					&= \sum \mathcal{M}'(h_2f \# 1_H)(\mathcal{M}'(id_X \# S(h_1))(m))\\
					&= \sum \mathcal{M}(h_2f)(h_1m). 
\end{align*}
\end{proof}

\begin{prop}\label{2.5}
Let $\mathcal{M}$ and $\mathcal{N}$ be right $\mathcal{C} \# H$-modules. Then, $Hom_{Mod\text{-}\mathcal C}(\mathcal{M},\mathcal{N})$ is a left $H$-module and its invariants are given by $Hom_{Mod\text{-}\mathcal{C}}(\mathcal{M},\mathcal{N})^H=Hom_{Mod\text{-}(\mathcal{C}\#H)}(\mathcal{M},\mathcal{N})$.
\end{prop}
\begin{proof} We have shown in Proposition \ref{2.3} that every right $\mathcal C\# H$-module is also a left $H$-equivariant
right $\mathcal C$-module. Accordingly, we use \eqref{3xee3/1} to give an $H$-module structure on $Hom_{Mod\text{-}\mathcal{C}}(\mathcal{M},\mathcal{N})$ by setting 
\begin{equation}\label{hin4}(h\cdot \eta)(X)(m):=\sum h_1\big(\eta(X)(S(h_2)m)\big)\quad  \forall ~ h \in H,~ X \in Ob(\mathcal{C}), ~m \in \mathcal{M}(X)
\end{equation} 
 for any $\eta \in Hom_{Mod\text{-}\mathcal{C}}(\mathcal{M},\mathcal{N})$. 
 
 \smallskip
Suppose now that $\eta \in Hom_{Mod\text{-}\mathcal{C}}(\mathcal{M},\mathcal{N})^H $. From  the proof of Proposition \ref{Pxp3.4}, it follows
that $\eta(X):\mathcal M(X)\longrightarrow\mathcal N(X)$ is $H$-linear for each $X\in Ob(\mathcal C)$. We need to show that $\eta \in Hom_{Mod\text{-}(\mathcal{C}\#H)}(\mathcal{M},\mathcal{N})$. For any $f:Y \longrightarrow X$ in $\mathcal{C}$, $h \in H$ and $m \in \mathcal{M}(X)$, we have
\begin{equation*}
\begin{array}{ll}
\eta(Y)\big(\mathcal{M}(f \# h)(m)\big) & = \eta(Y)(S^{-1}(h)\mathcal M(f)(m))\\
 &= \eta(Y)\left(\mathcal M(S^{-1}(h_1)f)(S^{-1}(h_2)m)\right)~~~~~~~~(\text{since $\mathcal M$ is $H$-equivariant})\\
&= \mathcal N(S^{-1}(h_1)f)\eta(X)(S^{-1}(h_2)m)\\
&=\mathcal N(S^{-1}(h_1)f)(S^{-1}(h_2)\eta(X)(m))\\
&= S^{-1}(h)(\mathcal N(f)(\eta(X)(m)))~~~~~~~~(\text{since $\mathcal N$ is $H$-equivariant})\\
&=\mathcal N(f\#h)(\eta(X)(m))
\end{array}
\end{equation*}

Conversely, let $\eta \in Hom_{Mod\text{-}(\mathcal{C}\#H)}(\mathcal{M},\mathcal{N})$. 
Using the $H$-linearity of $\eta(X)$ from Lemma \ref{2.2}, it is clear from \eqref{hin4} that $\eta \in Hom_{Mod\text{-}\mathcal{C}}(\mathcal{M},\mathcal{N})^H $.
\end{proof}

\begin{prop}\label{equicat}
Let $\mathcal C$ be a  left $H$-category. Then, the categories  $Mod\text{-}(\mathcal C\# H)$ and $(Mod\text{-}\mathcal C)_H^H$ are identical.  In particular, the category $(Mod\text{-}\mathcal C)_H^H$ of left $H$-equivariant  right $\mathcal C$-modules is a Grothendieck category.
\end{prop}
\begin{proof} The fact that  $Mod\text{-}(\mathcal C\# H)$ and $(Mod\text{-}\mathcal C)_H^H$ are identical follows from 
Propositions \ref{Pxp3.4}, \ref{2.3} and \ref{2.5}. 
Further, given any small preadditive category $\mathcal{E}$, it is well known that the category $Mod\text{-}\mathcal{E}$ is a Grothendieck category (see, for instance, \cite[Example V.2.2]{Sten}). Since $\mathcal C\# H$ is a small preadditive category, the result  follows.
\end{proof}

We denote by $\mathcal M\otimes_{\mathcal C}(\mathcal{C}\#H)$ the extension of a right $\mathcal{C}$-module $\mathcal{M}$ to a right $(\mathcal{C}\#H)$-module. For the general notion of {extension and restriction of scalars} in the case of modules over a category, see, for instance, \cite[$\S$ 4]{DF}. It follows from \cite[Proposition 19]{DF} that the extension of scalars is left adjoint to the restriction of scalars.

\smallskip

\begin{lemma}\label{ext} Let $\mathcal{M}$ be a right $\mathcal C$-module. Then,

\smallskip (1)
A right $(\mathcal{C}\#H)$-module $\mathcal{M} \otimes H$ may be obtained by setting
\begin{equation*}
\begin{array}{c}
(\mathcal{M} \otimes H)(X):=\mathcal{M}(X)\otimes H\\
\left((\mathcal{M} \otimes H)(f'\#h')\right)(m \otimes h):=\sum \mathcal{M}(h_1f')(m) \otimes h_2h'
\end{array}
\end{equation*}
for any $X \in Ob(\mathcal{C}\#H)$, $f'\#h' \in Hom_{\mathcal{C}\#H}(Y,X)$, $m \in \mathcal{M}(X)$ and $h,h' \in H$. 

\smallskip (2)
$\mathcal{M} \otimes H$ is isomorphic to $\mathcal{M} \otimes_{\mathcal{C}} (\mathcal{C}\#H)$ as objects in $Mod\text{-}(\mathcal{C}\#H)$.
\end{lemma}
\begin{proof}
(1) For any $f''\#h'' \in Hom_{\mathcal{C}\#H}(Z,Y)$, $f'\#h' \in Hom_{\mathcal{C}\#H}(Y,X)$, we have
\begin{equation*}
\begin{array}{ll}
\left((\mathcal{M} \otimes H)\left((f'\#h')(f''\#h'')\right)\right)(m \otimes h)&= \sum \left((\mathcal{M} \otimes H)(f'(h'_1f'')\#h'_2h'')\right)(m \otimes h)\\
&= \sum \mathcal{M}\left(h_1(f'(h'_1f'')\right)(m) \otimes h_2h'_2h''\\
&= \sum \mathcal{M}\left((h_1f')(h_2h'_1f'')\right)(m) \otimes h_3h'_2h'' \quad (\text{since $\mathcal{C}$ is a left $H$-category})\\
&= \sum \mathcal{M}(h_2h'_1f'')\mathcal{M}(h_1f')(m)  \otimes h_3h'_2h''\\
&= \sum \mathcal{M}({(h_2h')}_1f'')\mathcal{M}(h_1f')(m)  \otimes {(h_2h')}_2h''\\
&= \left((\mathcal{M}\otimes H)(f''\#h'')(\mathcal{M}\otimes H)(f'\#h')\right)(m \otimes h)
\end{array}
\end{equation*}
Further, $\left((\mathcal{M} \otimes H)(id_X \# 1_H)\right)(m \otimes h)=\mathcal{M}(h_1id_X)(m) \otimes h_2=m \otimes h$. Thus, $\mathcal{M} \otimes H \in Mod\text{-}(\mathcal{C}\#H)$.

\smallskip
(2) It may be easily checked that the assignment $\mathcal{M}\mapsto \mathcal{M}\otimes H$ defines a functor from $Mod$-$\mathcal{C}$ to $Mod$-$(\mathcal{C}\# H)$, which we denote by $({-})\otimes H$.
We will now show that the functor $({-})\otimes H: Mod$-$\mathcal{C}\longrightarrow Mod$-$(\mathcal{C}\# H)$ is the left adjoint to the restriction of scalars from $Mod$-$(\mathcal C\# H)$ to $Mod$-$\mathcal C$, i.e., there is a natural isomorphism $Hom_{Mod\text{-}(\mathcal C\# H)}(\mathcal{M}\otimes H, \mathcal{N})\cong Hom_{Mod\text{-}\mathcal{C}}(\mathcal{M},\mathcal{N})$. The result of $(2)$ will then follow from the uniqueness of adjoints. We define
$$\phi: Hom_{Mod\text{-}(\mathcal C\# H)}(\mathcal{M}\otimes H, \mathcal{N})\longrightarrow Hom_{Mod\text{-}\mathcal{C}}(\mathcal{M},\mathcal{N})$$ by setting $\phi(\eta)(X)(m):=\eta(X)(m\otimes 1_H)$ for any $\eta\in Hom_{Mod\text{-}(\mathcal C\# H)}(\mathcal{M}\otimes H, \mathcal{N})$, $X\in Ob(\mathcal{C})$ and $m\in \mathcal{M}(X)$. For any $f\in Hom_\mathcal{C}(X,Y)$ and $m'\in \mathcal{M}(Y)$, we have
\begin{align*}
\mathcal{N}(f)\phi(\eta)(Y)(m')= \mathcal{N}(f\# 1_H)\eta(Y)(m'\otimes 1_H)&=\left(\eta(X)(\mathcal{M}\otimes H)(f\# 1_H)\right)(m'\otimes 1_H)\\
&= \eta(X)(\mathcal{M}(f)(m')\otimes 1_H)= \phi(\eta)(X)(\mathcal{M}(f)(m'))
\end{align*}
Thus, $\phi(\eta) \in Hom_{Mod\text{-}\mathcal{C}}(\mathcal{M},\mathcal{N})$. We now check the injectivity of $\phi$. Let $\eta,\nu \in Hom_{Mod\text{-}(\mathcal C\# H)}(\mathcal{M}\otimes H, \mathcal{N})$ be such that $\phi(\eta)=\phi(\nu)$. Then, we have
\begin{equation*}
\begin{array}{ll}
\eta(X)(m\otimes h) =\eta(X)((\mathcal{M}\otimes H)(id_X \# h))(m\otimes 1_H)=\mathcal{N}(id_X \# h)\eta(X)(m\otimes 1_H)&=\mathcal{N}(id_X \# h)\phi(\eta)(X)(m)\\
&=\mathcal{N}(id_X \# h)\phi(\nu)(X)(m)\\
&=\nu(X)(m\otimes h) \\
\end{array}
\end{equation*}
for all $X\in Ob(\mathcal{C})$ and $m\otimes h \in \mathcal{M}(X)\otimes H$. This shows that $\eta=\nu$. Next, given any $\xi \in Hom_{Mod\text{-}\mathcal{C}}(\mathcal{M},\mathcal{N})$, we define $\eta(X):\mathcal{M}(X) \otimes H \longrightarrow \mathcal{N}(X)$ by 
$$\eta(X)(m\otimes h):= S^{-1}(h)\xi(X)(m)$$ for $X\in Ob(\mathcal{C}\#H)$ and $m\otimes h \in \mathcal{M}(X)\otimes H$. Then, for any $f'\# h' \in Hom_{Mod\text{-}(\mathcal{C}\#H)}(Y,X)$, we have
\begin{equation*}
\begin{array}{l}
\eta(Y)\left((\mathcal{M}\otimes H)(f'\#h')(m\otimes h)\right)
=\sum \eta(Y)\left( \mathcal{M}(h_1f')(m)\otimes h_2h'\right)\\
=\sum S^{-1}(h_2h')\left(\xi(Y)(\mathcal{M}(h_1f')(m))\right)\\
=\sum S^{-1}(h_2h')\left(\mathcal{N}(h_1f')\xi(X)(m)\right)\\
=\sum \mathcal{N}\left({(S^{-1}(h_2h'))}_2h_1f'\right)\left({(S^{-1}(h_2h'))}_1\xi(X)(m)\right)~~~~~~~~~~(\text{since~} \mathcal{N} \text{~is~} H\text{-equivariant})\\
=\sum \mathcal{N}\left(S^{-1}(h_1')S^{-1}(h_2)h_1f'\right)\left(S^{-1}(h_2')S^{-1}(h_3)\xi(X)(m)\right)\\
=\sum \mathcal{N}(S^{-1}(h_1')f')\left(S^{-1}(h_2')S^{-1}(h)\xi(X)(m)\right)\\
=S^{-1}(h')\mathcal{N}(f')\left(S^{-1}(h)\xi(X)(m)\right)~~~~~~~~~~~~~~~~~~~~~~~~~~~~~~~~(\text{since~} \mathcal{N} \text{~is~} H\text{-equivariant})\\
= \mathcal{N}(f'\#h')\left(S^{-1}(h)\xi(X)(m)\right)~~~~~~~~~~~~~~~~~~~~~~~~~~~~~~~~~~~~(\text{using \eqref{weq2.1}})\\
=\mathcal{N}(f'\#h')(\eta(X)(m\otimes h))
\end{array}
\end{equation*}
Hence, $\eta \in Hom_{Mod\text{-}(\mathcal C\# H)}(\mathcal{M}\otimes H, \mathcal{N})$. We also have $\phi(\eta)(X)(m)=\eta(X)(m\otimes 1_H)=\xi(X)(m)$, i.e., $\phi(\eta)=\xi$. Hence, $\phi$ is surjective. This proves the result.
\end{proof}

\begin{prop}\label{cor 2.8} (1) The extension of scalars from $Mod$-$\mathcal C$ to $Mod$-$(\mathcal C\# H)$ is exact. 

\smallskip
(2) Let $\mathcal{I}$ be an injective object in $Mod\text{-}(\mathcal C\# H)$. Then, $\mathcal{I}$ is also an injective object in $Mod\text{-}\mathcal{C}$.
\end{prop}
\begin{proof}
 Let $\mathcal{M},\mathcal{N}\in Mod$-$\mathcal C$ be such that $\phi: \mathcal{M}\longrightarrow \mathcal{N}$ is a monomorphism, i.e., $\phi(X): \mathcal{M}(X)\longrightarrow \mathcal{N}(X)$ is a monomorphism in $Vect_K$ for each $X\in Ob(\mathcal{C})$. Applying the isomorphism in Lemma \ref{ext},  $\big(\phi\otimes_{\mathcal{C}}(\mathcal{C}\#H)\big)(X)=(\phi\otimes H)(X): \mathcal{M}(X)\otimes H\longrightarrow \mathcal{N}(X)\otimes H$ is a monomorphism. Since extension of scalars is a left adjoint, it already preserves colimits. This proves (1). The result of (2) now follows
from \cite[Tag 015Y]{SP}. 
\end{proof}

\begin{lemma}\label{tenpro}
Let $M \in H\text{-}Mod$ and let $\mathcal{N} \in Mod$-$(\mathcal C\# H)$. Then, a right $(\mathcal C\# H)$-module $M\otimes \mathcal{N}$ can be defined by setting 
$$(M\otimes \mathcal{N})(X) := M \otimes \mathcal{N}(X)\qquad 
(M \otimes \mathcal{N})(f)(m\otimes n) := m\otimes \mathcal{N}(f)(n)$$
for any $X \in \text{Ob}(\mathcal{C}),~ f\in Hom_{\mathcal{C}}(Y,X)$ and $m \otimes n \in M\otimes \mathcal{N}(X)$.
\end{lemma}
\begin{proof}
It is clear that $M\otimes \mathcal{N} \in Mod$-$\mathcal{C}$. Now for each $X \in Ob(\mathcal{C})$, the $K$-vector space $M\otimes \mathcal{N}(X)$ has a left $H$-module structure given by $$h(m \otimes n) := \sum h_1m \otimes h_2n \qquad\forall ~ h\in H$$ It may be easily verified that $M\otimes \mathcal{N}$ is an $H$-equivariant right $\mathcal{C}$-module under this action. Therefore, $M\otimes\mathcal{N}  \in Mod$-$(\mathcal C\# H)$ by Proposition \ref{2.3}.
\end{proof}
Given any $\mathcal{N} \in Mod$-$(\mathcal C\# H)$, let $({-})\otimes \mathcal{N}: H\text{-}Mod\longrightarrow Mod$-$(\mathcal C\# H)$ denote the functor which takes any $M\in H$-$Mod$ to $M\otimes \mathcal{N}\in Mod$-$(\mathcal C\# H)$.
\begin{prop}\label{2.7}
Let $\mathcal{N}, \mathcal{P} \in Mod\text{-}(\mathcal{C}\# H)$ and let $M \in H\text{-}Mod$. Then, we have a natural isomorphism
$$\phi: Hom_{Mod\text{-}(\mathcal{C}\#H)}\left(M \otimes \mathcal{N},\mathcal{P}\right) \longrightarrow Hom_{H\text{-}Mod}\left(M, Hom_{Mod\text{-}\mathcal{C}}(\mathcal{N},\mathcal{P})\right)$$
given by $\big(\phi(\eta)(m)\big)(X)(n) := \eta(X)(m \otimes n)$ for each  $X \in Ob(\mathcal{C})$ and $m\in M, n \in \mathcal{N}(X)$.
\end{prop}
\begin{proof}
Let $\eta \in Hom_{Mod\text{-}(\mathcal{C}\#H)}\left(M \otimes \mathcal{N},\mathcal{P}\right)$. It may be checked that $\phi(\eta)(m) \in Hom_{Mod\text{-}\mathcal{C}}(\mathcal{N},\mathcal{P})$ for every $m \in M$. We now verify that $\phi(\eta)$ is $H$-linear, i.e., for $h\in H$:
\begin{equation*}
\begin{array}{ll}
\Big(h\big(\phi(\eta)(m)\big)\Big)(X)(n) &=\sum h_1\Big(\phi(\eta)(m)(X)\big(S(h_2)n\big)\Big)\ \ \ \ \ \ \ (\text{using Proposition}~ \ref{2.5})\\
&= \sum h_1\Big(\eta(X)\big(m \otimes S(h_2)n\big)\Big)\\
&=\sum \eta(X)\left(h_1m \otimes h_2S(h_3)n\right)\ \ \ \ \ \ \ (\textnormal{since\ } \eta(X) \textnormal{\ is\ } H\textnormal{-linear}~ \text{by Lemma \ref{2.2}})\\
&=\eta(X)(hm \otimes n)= \Big(\phi(\eta)(hm)\Big)(X)(n)
\end{array}
\end{equation*}
Clearly, $\phi$ is injective. For $f\in Hom_{H\text{-}Mod}\big(M, Hom_{Mod\text{-}\mathcal{C}}(\mathcal{N},\mathcal{P})\big)$, we consider  $\nu \in Hom_{Mod\text{-}(\mathcal{C}\#H)}(M \otimes \mathcal{N},\mathcal{P})$ determined by 
\begin{equation}\label{trx3.1}\nu(X)(m\otimes n) :=f(m)(X)(n)
\end{equation}
for each $X\in Ob(\mathcal{C}), n\in \mathcal{N}(X)$ and $m \in M$. We first check that $\nu(X):M \otimes \mathcal{N}(X) \longrightarrow \mathcal{P}(X)$ is $H$-linear for every $X\in Ob(\mathcal{C})$,  i.e., for $h\in H$:
\begin{equation*}
\begin{array}{ll}
\nu(X)\left(h(m\otimes n)\right)&=\sum\nu(X)(h_1m\otimes h_2n)\\ &
=\sum f(h_1m)(X)(h_2n)\\
&=\sum (h_1f(m))(X)(h_2n)~~~~~~~~~~~(\textnormal{since $f$ is $H$-linear})\\
&=\sum h_1\left(f(m)(X)\right)\left(S(h_2)h_3n\right)\\&=h\left(\nu(X)(m\otimes n)\right)
\end{array}
\end{equation*}
Using the fact that $f(m) \in Hom_{Mod\text{-}\mathcal{C}}(\mathcal{N},\mathcal{P})$ for each $m \in M$, it may now be verified  that $\nu \in Hom_{Mod\text{-}\mathcal{C}}(M \otimes \mathcal{N},\mathcal{P})$. From the equivalence
of categories in Proposition \ref{equicat}, it follows that $\nu \in  Hom_{Mod\text{-}(\mathcal{C}\#H)}\left(M\otimes \mathcal{N},\mathcal{P}\right)$.  From \eqref{trx3.1}, it is also clear that $\phi(\nu)=f$. 
\end{proof}

\begin{corollary}\label{2.8}
If $\mathcal{I}$ is an injective object in $Mod$-$(\mathcal C\# H)$, then $Hom_{Mod\text{-}\mathcal{C}}(\mathcal{N},\mathcal{I})$ is an injective object in $H$-$Mod$ for any $\mathcal{N} \in Mod$-$(\mathcal C\# H)$.
\end{corollary}
\begin{proof}
From Proposition \ref{2.7}, we know that the functor $(-)\otimes \mathcal{N}:H$-$Mod
\longrightarrow Mod$-$(\mathcal C\# H)$ is a left adjoint and therefore preserves colimits. Further, given a  monomorphism $M_1\hookrightarrow M_2$ in $H$-$Mod$, it is clear from the definition in Lemma \ref{tenpro} that $M_1\otimes\mathcal N
\longrightarrow M_2\otimes\mathcal N$ is a monomorphism  in $Mod$-$(\mathcal C\# H)$.  Hence, $(-)\otimes \mathcal{N}:H$-$Mod
\longrightarrow Mod$-$(\mathcal C\# H)$ is exact. 
 As such, 
its right adjoint $Hom_{Mod\text{-}\mathcal{C}}(\mathcal{N},\_\_): Mod$-$(\mathcal C\# H)\longrightarrow H$-$Mod$
preserves injectives. 
\end{proof}

We denote by $(-)^H$ the functor from $H\text{-}Mod$ to $Vect_K$ that takes $M$ to $M^H=\{\mbox{$m\in M$ $\vert$ 
$hm=\varepsilon(h)m$ $\forall$ $h\in H$}\}$. We now recall from Proposition \ref{2.5} that we have an isomorphism
\begin{equation*}
Hom_{Mod\text{-}\mathcal{C}}(\mathcal{M},\mathcal{N})^H\cong Hom_{Mod\text{-}(\mathcal{C}\#H)}(\mathcal{M},\mathcal{N})
\end{equation*}
for any $\mathcal M$, $\mathcal N\in Mod$-$(\mathcal C\# H)$.  At the level of the derived $Hom$ functors, this leads to the following  spectral sequence.

\begin{theorem}\label{Tc3.15}
Let $\mathcal{M},~ \mathcal{N} \in Mod\text{-}(\mathcal{C}\#H)$. Then, there exists a first quadrant spectral sequence: $$R^{p}(-)^H\left(\textnormal{Ext}^q_{Mod\text{-}\mathcal{C}}(\mathcal{M},\mathcal N)\right)\Rightarrow \left(R^{p+q}Hom_{Mod\text{-}(\mathcal{C}\#H)}(\mathcal{M},{-})\right)(\mathcal{N})$$ 
\end{theorem}
\begin{proof}
We consider the functors $\mathscr{F}:= Hom_{Mod\text{-}\mathcal{C}}(\mathcal{M},-): Mod\text{-}(\mathcal{C}\#H)\longrightarrow H\text{-}{Mod}$ and $\mathscr{G}:= ({-})^{H}: H{\text{-}}Mod \longrightarrow Vect_K$. We notice that
$Mod$-$(\mathcal C\# H)$, $H$-$Mod$ and $Vect_K$ are all Grothendieck categories. 
From  Corollary \ref{2.8}, we know that $\mathscr{F}$ preserves injectives. Using Proposition \ref{2.5}, we see that the functor $(\mathscr{G}\circ\mathscr{F}):Mod\text{-}(\mathcal{C}\#H)\longrightarrow Vect_{K}$ is given by $(\mathscr{G}\circ\mathscr{F})(\mathcal{N})= Hom_{Mod\text{-}\mathcal{C}}(\mathcal{M},\mathcal{N})^H = Hom_{Mod\text{-}(\mathcal{C}\#H)}(\mathcal{M},\mathcal{N})$. The result now follows from the  Grothendieck spectral sequence for
composite functors (see \cite{Grothen}). 
\end{proof}

\section{$H$-locally finite modules and cohomology}
We recall the definition of $H$-locally finite modules from \cite{Gue}. For $M$ $\in H$-$Mod$ and $m \in M$, let $Hm$ be the $H$-submodule of $M$ spanned by the elements $hm$ for $h \in H$. Consider
\begin{equation*}
M^{(H)}:=\{m \in M ~|~ Hm~ \text{is finite dimensional as a $K$-vector space}~\}
\end{equation*}

Clearly, $M^{(H)}$ is an $H$-submodule of $M$. An $H$-module $M$ is said to be $H$-locally finite if $M^{(H)}=M$. The full subcategory of $H$-$Mod$ whose objects are $H$-locally finite $H$-modules will be denoted by $H$-$mod$.

\smallskip
By Proposition \ref{2.5}, 
$Hom_{Mod\text{-}\mathcal{C}}(\mathcal{N},\mathcal{P})$ is an $H$-module for any $\mathcal{N},\mathcal{P} \in Mod\text{-}(\mathcal{C}\#H)$. We set $$\mathcal{L}_{Mod\text{-}\mathcal{C}}(\mathcal{N},\mathcal{P}):=Hom_{Mod\text{-}\mathcal{C}}{(\mathcal{N},\mathcal{P})}^{(H)}$$
Clearly, this defines a functor $\mathcal{L}_{Mod\text{-}\mathcal{C}}(\mathcal{N},{-}):Mod\text{-}(\mathcal{C}\# H) \longrightarrow H$-$mod$ for every $\mathcal{N}\in Mod\text{-}(\mathcal{C}\#H)$.

\begin{prop}\label{lem 3.4}
Let $\mathcal{N} \in Mod\text{-}(\mathcal{C}\#H)$. Then, the functor $\mathcal{L}_{Mod\text{-}\mathcal{C}}(\mathcal{N},{-}):Mod\text{-}(\mathcal{C}\# H) \longrightarrow H$-$mod$ is right adjoint to the functor $({-})\otimes \mathcal{N}:H$-$mod \longrightarrow Mod\text{-}(\mathcal{C}\# H)$, i.e., we have natural isomorphisms
\begin{equation*}
Hom_{Mod\text{-}(\mathcal{C}\#H)}(M \otimes \mathcal{N},\mathcal{P}) \cong Hom_{H\text{-}mod}\big(M,\mathcal{L}_{Mod\text{-}\mathcal{C}}(\mathcal{N},\mathcal{P})\big)
\end{equation*}
for all $\mathcal{P}\in Mod\text{-}(\mathcal{C}\#H)$ and $M\in H$-$mod$.
\end{prop}
\begin{proof}
Let $\phi:Hom_{Mod\text{-}(\mathcal{C}\#H)}(M \otimes \mathcal{N},\mathcal{P}) \longrightarrow Hom_{H\text{-}Mod}\big(M,Hom_{Mod\text{-}\mathcal{C}}(\mathcal{N},\mathcal{P})\big)$ be the isomorphism as in Proposition \ref{2.7}. Let $\eta:M \otimes \mathcal{N} \longrightarrow \mathcal{P}$ be a morphism in $Mod\text{-}(\mathcal{C}\#H)$. It follows that $H\phi(\eta)(m)$ is finite dimensional for each $m \in M$ by observing that $\phi(\eta)$ is $H$-linear and that $M$ is $H$-locally finite. Since $H$-$mod$ is a full subcategory of $H$-$Mod$, we have $Hom_{Mod\text{-}(\mathcal{C}\#H)}(M \otimes \mathcal{N},\mathcal{P}) \cong Hom_{H\text{-}Mod}\big(M,\mathcal{L}_{Mod\text{-}\mathcal{C}}(\mathcal{N},\mathcal{P})\big)\cong  Hom_{H\text{-}mod}\big(M,\mathcal{L}_{Mod\text{-}\mathcal{C}}(\mathcal{N},\mathcal{P})\big).$
\end{proof}

For any $\mathcal M\in Mod$-$(\mathcal C\# H)$, we can now consider the functor
\begin{equation}
\mathcal L_{Mod\text{-}\mathcal C}(\mathcal M,-):Mod\text{-}(\mathcal C\# H)\longrightarrow H\text{-}mod\qquad \mathcal{ N}\mapsto \mathcal L_{Mod\text{-}\mathcal{C}}(\mathcal{M},\mathcal{N})
\end{equation} Since $Mod\text{-}(\mathcal C\# H)$ is a Grothendieck category, we obtain derived functors
$\mathbf{R}^p\mathcal L_{Mod\text{-}\mathcal C}(\mathcal M,-):Mod\text{-}(\mathcal C\# H)\longrightarrow H\text{-}mod$, $p\geq 0$. We use the boldface notation to distinguish these from the functors ${R}^p\mathcal L_{Mod\text{-}\mathcal C}(\mathcal M,-)$
that will appear later in the proof of Proposition \ref{3.14} as derived functors of a  restriction of $\mathcal L_{Mod\text{-}\mathcal C}(\mathcal M,-)$. 

\begin{theorem}\label{Tb4.19}
Let $\mathcal{M},~ \mathcal{N} \in Mod$-$(\mathcal{C} \#H)$. We consider the functors 
\begin{equation*}
\begin{array}{c}\mathscr F=Hom_{Mod\text{-}\mathcal{C}}(\mathcal{M},-): {Mod}\text{-}(\mathcal{C}\#H)\longrightarrow H\text{-}{Mod}
\qquad \mathcal{ N}\mapsto Hom_{Mod\text{-}\mathcal{C}}(\mathcal{M},\mathcal{N})\\
\mathscr G= ({-})^{(H)}:H\text{-}{Mod} \longrightarrow H\text{-}mod \qquad M\mapsto M^{(H)}\\
 \end{array}
 \end{equation*} Then, we have the following spectral sequence
$$R^p{(-)}^{(H)}({\textnormal{Ext}}^q_{Mod\text{-}\mathcal{C}}(\mathcal{M},\mathcal{N}))\Rightarrow \left(\mathbf{R}^{p+q}\mathcal{L}_{Mod\text{-}\mathcal{C}}(\mathcal{M},-)\right)(\mathcal{N})$$
\end{theorem}
\begin{proof}
We have $(\mathscr{G}\circ\mathscr{F})(\mathcal{N})=Hom_{Mod\text{-}\mathcal{C}}(\mathcal{M},\mathcal{N})^{(H)}=\mathcal{L}_{Mod\text{-}\mathcal{C}}(\mathcal{M},\mathcal{N})$. By definition, 
\begin{equation*}
R^q\mathscr F(N)=H^q(\mathscr F(\mathcal I^*))=H^q(Hom_{Mod\text{-}\mathcal{C}}(\mathcal{M},\mathcal I^*))
\end{equation*}  where   $\mathcal{I}^*$ is an injective resolution of $\mathcal{N}$ in $Mod$-$(\mathcal{C}\#H)$. By Corollary \ref{cor 2.8}, injectives in $Mod$-$(\mathcal C\# H)$ are also injectives in $Mod$-$\mathcal C$. Hence, $R^q\mathscr F(\mathcal N)=\textnormal{Ext}_{Mod\text{-}\mathcal{C}}^q(\mathcal{M},\mathcal{N})$. For any injective $\mathcal{I}$ in $Mod$-$(\mathcal C\# H)$, we know that $\mathscr{F}(\mathcal{I})$ is injective in $H$-$Mod$ by Corollary \ref{2.8}. Since the category $H\text{-}Mod$ has enough injectives, the result now follows from Grothendieck spectral sequence for composite functors (see \cite{Grothen}). 
\end{proof}
 
\begin{definition} Let $\mathcal{C}$ be a left $H$-category. 
\begin{itemize}
\item[(1)] $\mathcal{C}$ is said to be $H$-locally finite if the $H$-module $Hom_\mathcal{C}(X,Y)$ is $H$-locally finite, i.e., $Hom_\mathcal{C}(X,Y)^{(H)}=Hom_\mathcal{C}(X,Y)$, for all $X,Y \in Ob(\mathcal{C})$.

\item[(2)] Let $\mathcal{M} \in Mod\text{-}(\mathcal{C}\# H)$. Then, $\mathcal{M}$ is said to be $H$-locally finite if the $H$-module $\mathcal{M}(X)$ is $H$-locally finite, i.e. $\mathcal{M}(X)^{(H)}=\mathcal{M}(X)$, for each $X \in \text{Ob}(\mathcal{C})$. The full subcategory of $Mod$-$(\mathcal{C}\# H)$ whose objects are $H$-locally finite right $(\mathcal{C}\# H)$-modules will be denoted by $mod$-$(\mathcal{C}\# H)$. 
\end{itemize}
\end{definition} 

If $M$, $M'\in H$-$Mod$, we know that $H$ acts diagonally on their tensor product $M\otimes M'$ over $K$, i.e., 
$h(m\otimes m')=\sum h_1m\otimes h_2m'$ for $h\in H$, $m\in M$ and $m'\in M'$. In particular, if $M$, $M'\in H$-$mod$, it follows
that $M\otimes M'\in H$-$mod$. 
Accordingly, if $\mathcal{N}\in mod\text{-}(\mathcal{C}\# H)$ and $M\in H$-$mod$, it is clear from the definition 
in Lemma \ref{tenpro} that  $M\otimes \mathcal{N}\in mod\text{-}(\mathcal{C}\# H)$.

\begin{corollary}\label{lem 3.5}
Let $\mathcal{N} \in mod\text{-}(\mathcal{C}\#H)$. Then, the functor $\mathcal{L}_{Mod\text{-}\mathcal{C}}(\mathcal{N},{-}):mod\text{-}(\mathcal{C}\# H) \longrightarrow H$-$mod$ is right adjoint to the functor $({-})\otimes \mathcal{N}:H$-$mod \longrightarrow mod\text{-}(\mathcal{C}\# H)$, i.e., we  have natural isomorphisms
\begin{equation*}
Hom_{mod\text{-}(\mathcal{C}\#H)}(M \otimes \mathcal{N},\mathcal{P}) \cong Hom_{H\text{-}mod}\big(M,\mathcal{L}_{Mod\text{-}\mathcal{C}}(\mathcal{N},\mathcal{P})\big)
\end{equation*}
for all $\mathcal{P}\in mod\text{-}(\mathcal{C}\#H)$ and $M\in H$-$mod$.
\end{corollary}

\begin{proof}
This follows from Proposition \ref{lem 3.4} because $mod$-$(\mathcal C\# H)$ is a full subcategory 
of 
$Mod$-$(\mathcal C\# H)$. 
\end{proof}

\begin{lemma}\label{lem4.9x} Let $\mathcal C$ be a left $H$-category. Given $X\in Ob(\mathcal C)$, consider
the representable functor ${\bf{h}}_X\in Mod$-$\mathcal C$. Then,  
 the right $\mathcal C$-module ${\bf{h}}_X$ is also a right $(\mathcal{C}\#H)$-module.

\end{lemma}

\begin{proof} For each $Y\in Ob(\mathcal{C})$, we have ${\bf h}_X(Y)= Hom_{\mathcal{C}}(Y,X)$. Since $\mathcal{C}$ is a left $H$-category, ${\bf h}_X(Y)$ has a left $H$-module structure. For any $f\in Hom_\mathcal{C}(Z,Y)$, $g \in {\bf h}_X(Y)$ and $h \in H$, we have
$$ h\left({\bf h}_X(f)(g)\right)= h(gf)= \sum h_1(g)h_2(f)=\sum{\bf h}_X(h_2f)(h_1g)
$$
Thus, ${\bf h}_X$ is a left $H$-equivariant right $\mathcal{C}$-module. Hence, ${\bf h}_X\in Mod\text{-}(\mathcal{C}\#H)$ by Proposition \ref{equicat}.
\end{proof}

\begin{lemma}\label{inj}
(1) If   $\mathcal{I}$ is an injective in $Mod$-$(\mathcal C\# H)$, then $\mathcal{L}_{Mod\text{-}\mathcal{C}}(\mathcal{N},\mathcal{I})$ is an injective in $H$-$mod$ for any $\mathcal{N} \in Mod$-$(\mathcal C\# H)$.

\smallskip
(2) If $\mathcal{I}$ is an injective in $mod\text{-}(\mathcal{C}\#H)$, then 
\begin{itemize} 
\item[(i)] $\mathcal{L}_{Mod\text{-}\mathcal{C}}(\mathcal{N},\mathcal{I})$ is an injective in $H$-$mod$ for any $\mathcal{N} \in mod\text{-}(\mathcal{C}\#H)$.
\item[(ii)] Let $\mathcal{C}$ be $H$-locally finite. Then, for each $X \in \text{Ob}(\mathcal C \#H)$, $\mathcal{I}(X)$ is an injective in $H$-$mod$.
\end{itemize}

\end{lemma}
\begin{proof}
(1)  The functor $\mathcal{L}_{Mod\text{-}\mathcal{C}}(\mathcal{N},{-}):Mod\text{-}(\mathcal{C}\# H) \longrightarrow H$-$mod$ is right adjoint to the functor $(-) \otimes \mathcal{N}:H\text{-}mod \longrightarrow Mod\text{-}(\mathcal{C} \# H)$ by Proposition \ref{lem 3.4}. Further, the functor $(-) \otimes \mathcal{N}$ always preserves monomorphisms (see the proof of Corollary \ref{2.8}). The result now follows from \cite[Tag 015Y]{SP}.

\smallskip
(2)  The proof of (i) is exactly the  same as that of (1) except that we use Corollary \ref{lem 3.5} in place of Proposition \ref{lem 3.4}.
% follows again from Lemma \ref{lem 3.4} and from the exactness of the functor $(-) \otimes \mathcal{N}:H\text{-}mod \longrightarrow mod\text{-}(\mathcal{C} \# H)$.
To prove (ii), we consider for each $X\in Ob(\mathcal C)$ the representable functor ${\bf h}_X\in Mod$-$\mathcal C$. 
Using Lemma \ref{lem4.9x}, we know that ${\bf{h}}_X \in Mod$-$(\mathcal{C}\#H)$. Further,  since $\mathcal{C}$ is $H$-locally finite, we see that ${\bf{h}}_X \in mod$-$(\mathcal{C}\#H)$. Using (i), we have $\mathcal{L}_{Mod\text{-}\mathcal{C}}({\bf{h}}_X,\mathcal{I})$ is injective in $H$-mod. Finally, by Yoneda lemma, we have $\mathcal{L}_{Mod\text{-}\mathcal{C}}({\bf{h}}_X,\mathcal{I})=Hom_{Mod\text{-}\mathcal C}({\bf h}_X,
\mathcal I)^{(H)}=\mathcal I(X)^{(H)}=\mathcal{I}(X)$. 
\end{proof}

\begin{lemma}
Let $\mathcal{C}$ be an $H$-locally finite category. Then, for any $\mathcal{M}$ in $Mod$-$(\mathcal C\# H)$, we may obtain an object $\mathcal{M}^{(H)} \in mod\text{-}(\mathcal{C}\#H)$ by setting
\begin{equation*}
\begin{array}{c}
\mathcal{M}^{(H)}(X):=\mathcal M(X)^{(H)}=\{m \in \mathcal{M}(X)~ |~ Hm\ \text{is finite dimensional}\}\\
\mathcal{M}^{(H)}(f\#h)(m):=\mathcal M(f\# h)(m)%\mathcal{M}(f)(hm)
\end{array}
\end{equation*}
for any $f \# h \in Hom_{\mathcal{C}\#H}(Y,X)$ and $m \in \mathcal{M}^{(H)}(X)$.
\end{lemma}
\begin{proof}
We need to verify that $\mathcal M(f\# h):\mathcal M(X)\longrightarrow \mathcal M(Y)$ restricts 
to a map $\mathcal M(X)^{(H)}\longrightarrow \mathcal M(Y)^{(H)}$. For this, we consider $m\in \mathcal M(X)^{(H)}$. Since
$\mathcal M\in Mod$-$(\mathcal C\# H)$ may be treated as a left $H$-equivariant right $\mathcal C$-module as in 
Proposition \ref{2.3}, we obtain
\begin{equation}\label{eq4.1zz}
h'\mathcal M(f\# h)(m)=h'S^{-1}(h)(\mathcal M(f)(m))=\sum \mathcal M(h'_2S^{-1}(h_1)f)(h'_1S^{-1}(h_2)m)
\end{equation} for any $h'\in H$.  Since the category $\mathcal C$ is $H$-locally finite and $m\in \mathcal M^{(H)}(X)$, it is clear 
from \eqref{eq4.1zz} that $\mathcal M(f\# h)(m)\in \mathcal M(Y)^{(H)}=\mathcal M^{(H)}(Y)$. This proves the result. 
\end{proof}

\begin{prop}\label{adj}~
\begin{itemize}
\item  [(1)] Let $({-})^{(H)}: H{\text{-}}Mod \longrightarrow H{\text{-}}mod$ be the functor $N\mapsto N^{(H)}$. Then $({-})^{(H)}$ is right adjoint to the forgetful functor from the category $H{\text{-}}mod $ to the category $H{\text{-}}Mod $, i.e.,
 we have natural isomorphisms
\begin{equation*}
 Hom_{H{\text{-}}Mod}(M,N)\cong Hom_{H{\text{-}}mod}(M,N^{(H)}) 
\end{equation*}
for any $M \in H$-$mod$ and $N\in H$-$Mod$.
\item [(2)] Let $({-})^{(H)}: Mod\text{-}(\mathcal{C}\#H) \longrightarrow mod\text{-}(\mathcal{C}\#H)$ be the functor $\mathcal{N} \mapsto \mathcal{N}^{(H)}$. Then $({-})^{(H)}$ is right adjoint to the forgetful functor from the category $mod\text{-}(\mathcal{C}\#H)$ to the category $Mod\text{-}(\mathcal{C}\#H)$, i.e., we have natural isomorphisms
\begin{equation*}
Hom_{Mod\text{-}(\mathcal{C}\#H)}(\mathcal{M},\mathcal{N})\cong Hom_{mod\text{-}(\mathcal{C}\#H)}(\mathcal{M},\mathcal{N}^{(H)}) 
\end{equation*}
for any $\mathcal{M} \in mod$-$(\mathcal C\# H)$ and $\mathcal N\in Mod$-$(\mathcal C\# H)$.
\end{itemize}
\end{prop}
\begin{proof}

(1) Given any $M \in H$-$mod$, $N \in H$-$Mod$ and an $H$-module morphism $\phi:M \longrightarrow N$, it is clear that $\phi(m) \in N^{(H)}$ for all $m \in M$.

\smallskip
(2) Let  $\mathcal{M} \in mod\text{-}(\mathcal{C}\#H)$. By Lemma \ref{2.2}, a morphism $\eta:\mathcal{M} \longrightarrow \mathcal{N}$ in $Mod$-$(\mathcal C\# H)$ induces $H$-linear morphisms $\eta(X):\mathcal{M}(X)\longrightarrow \mathcal{N}(X)$  for each
$X\in Ob(\mathcal C)$. Since $\mathcal M(X)$ is $H$-locally finite, each $\eta(X)$ can be written as a morphism 
$\mathcal M(X)\longrightarrow\mathcal N(X)^{(H)}$. The result is now clear. 
\end{proof}

\begin{lemma} Let $\mathcal{C}$ be $H$-locally finite. Then,
\begin{itemize}
\item[(1)] The category $mod\text{-}(\mathcal{C}\#H)$ is abelian.
\item[(2)] If $\mathcal{I}$ is an injective in $Mod\text{-}(\mathcal{C}\#H)$, then $\mathcal{I}^{(H)}$ is an injective in $mod\text{-}(\mathcal{C}\#H)$.
\item[(3)] The category $mod\text{-}(\mathcal{C}\#H)$ has enough injectives.
\end{itemize}
\end{lemma}
\begin{proof}
(1) Since $H$-$mod$ is closed under kernels and cokernels, it is clear that  the subcategory $mod\text{-}(\mathcal{C}\#H)$ of the abelian category $Mod$-$(\mathcal C\# H)$ is closed under kernels and cokernels. Also, since products and coproducts of finitely many objects $\mathcal{M}_i$ in $mod\text{-}(\mathcal{C}\#H)$ are given by
$$(\prod \mathcal{M}_i)(X) = (\coprod \mathcal{M}_i)(X)= \bigoplus_i \mathcal{M}_i(X)$$ for $X\in Ob(\mathcal{C}\#H)$,  it follows that finite products and coproducts exist and coincide in  $mod\text{-}(\mathcal{C}\#H)$.  Thus, the category $mod\text{-}(\mathcal{C}\#H)$ is abelian.

\smallskip
(2) Since the functor $({-})^{(H)}$ is right adjoint to the forgetful functor in Proposition \ref{adj}(2) and the forgetful functor always preserves monomorphisms, this result follows from \cite[Tag 015Y]{SP}.

\smallskip
(3) Since $Mod$-$(\mathcal C\# H)$ is a Grothendieck category, it has enough injectives. The result is now clear  from (2). 
\end{proof}

We now recall the notions of free, finitely generated and noetherian modules over a category from \cite[$\S$ 3]{Mit1} and \cite{Mit2}. Given $\mathcal M\in Mod$-$\mathcal C$, we set $el(\mathcal M):=\underset{X\in Ob(\mathcal C)}{\coprod}
\mathcal M(X)$ to be the collection of all elements of $\mathcal M$. If $m\in el(\mathcal M)$ lies in $\mathcal M(X)$, we write
$|m|=X$. 
 
\begin{definition}
Let $\mathcal{C}$ be a small preadditive category and let $\mathcal{M}\in Mod$-$\mathcal C$.
\begin{itemize}
\item[(i)] A family of elements $\{m_i \in el(\mathcal M)\}_{i\in I}$ is said to generate $\mathcal{M}$ if every element $y\in el(\mathcal M)$  can be expressed as $y=\sum_{i\in I} \mathcal{M}(f_i)(m_i)$ for some $f_i \in Hom_\mathcal{C}(|y|,|m_i|)$, where
all but a finite number of the $f_i$ are zero. Equivalently, the
family $\{m_i \in el(\mathcal M)\}_{i\in I}$  is said to generate $\mathcal M$ if the induced morphism $$ \eta: \bigoplus_{i\in I} {\bf h}_{|m_i|} \longrightarrow \mathcal{M}$$ which takes $(0,...,0,id_{|m_i|},0,...,0)$ to $m_i$ is an epimorphism. The family is said to be a basis for $\mathcal{M}$ if $\eta$ is an isomorphism. The module $\mathcal{M}$ is said to be finitely generated (resp. free)  if it has a finite set of generators (resp. a basis). 

\item[(ii)] The module $\mathcal{M}$ is called noetherian if it satisfies the ascending chain condition on submodules.  The category $\mathcal{C}$ is said to be right noetherian if ${\bf h}_X \in Mod\text{-}\mathcal{C}$ is noetherian for each $X\in Ob(\mathcal{C})$.
\end{itemize} 
\end{definition}

\begin{prop}\label{fg}~  
Let $\mathcal{C}$ be a left $H$-category.  
An object $\mathcal{M}\in mod$-$(\mathcal{C}\#H)$ is finitely generated in $Mod$-$(\mathcal C\# H)$ if and only if there exists a finite dimensional $V \in H$-$Mod$ and an epimorphism $$ V \otimes\left(\bigoplus_{i=1}^n {\bf h}_{X_i} \right)\longrightarrow \mathcal{M}$$ in $Mod$-$(\mathcal C\# H)$ for finitely many objects $\{X_i\}_{1\leq i\leq n}$ in $\mathcal{C}$,  where each ${\bf h}_{X_i}$ is  viewed as an object in $Mod$-$(\mathcal{C}\#H)$.
\end{prop}
\begin{proof}

 Let $\mathcal{M}\in mod$-$(\mathcal{C}\#H)$ be finitely generated in $Mod$-$(\mathcal C\# H)$. We consider
 a finite generating family $\{m_i\in el(\mathcal M)\}_{1\leq i\leq n}$ for $\mathcal M$.  Since $\mathcal{M}\in mod$-$(\mathcal{C}\#H)$, each $\mathcal{M}(|m_i|)$ is $H$-locally finite and hence the $H$-module $V:=\underset{i=1}{\overset{n}{\bigoplus}} Hm_i$ is finite dimensional. For each $Y \in \text{Ob}(\mathcal{C}\#H)$, we consider the morphism determined by setting
 \begin{equation*}
\eta(Y): V\otimes \left(\underset{i=1}{\overset{n}{\bigoplus}} {\bf{h}}_{|m_i|}(Y) \right) \longrightarrow \mathcal{M}(Y) \qquad 
 hm_i \otimes (f_1,\dots,f_n) \mapsto  \mathcal{M}(f_i\# h)(m_i)
\end{equation*}
It is easy to check that $\eta$ is a morphism in $Mod$-$(\mathcal C\# H)$ and $\eta(Y)$ is an epimorphism for all $Y \in \text{Ob}(\mathcal{C}\#H)$. Conversely, let $\{v_1,\ldots,v_k\}$ be a basis of a finite dimensional $H$-module $V$ and $X_1,\ldots,X_n$ be finitely many objects in $\mathcal{C}$ such that there is an epimorphism  $$\eta: V \otimes\left(\bigoplus_{i=1}^n  {\bf{h}}_{X_i}\right) \longrightarrow \mathcal{M}$$ in $Mod$-$(\mathcal C\# H)$.  It may be verified that the elements 
$\{m_{ij}=\eta(X_i)\left(v_j \otimes \text{id}_{X_i}\right) \in \mathcal{M}(X_i)\}_{1\leq i\leq n, 1\leq j\leq k}$ give a family of generators for $\mathcal{M}$.
\end{proof}

\begin{corollary}\label{CHfgC}
An object $\mathcal{M}$ in $mod$-$(\mathcal{C}\#H)$ is finitely generated in $Mod$-$(\mathcal C\# H)$ if and only if $\mathcal{M}$ is finitely generated in $Mod$-$\mathcal C$.
\end{corollary}
\begin{proof}
Let $\mathcal{M}\in mod$-$(\mathcal{C}\#H)$ be finitely generated in $Mod$-$\mathcal C$. Then, there is a finite family $\{m_i \in el(\mathcal M)\}_{i\in I}$ of elements of $\mathcal M$ such that every $y\in el(\mathcal M)$  can be expressed as $y=\sum_{i\in I} \mathcal{M}(f_i)(m_i)$ for some $f_i \in Hom_\mathcal{C}(|y|,|m_i|)$. Then, $y=\sum_{i\in I} \mathcal{M}(f_i\# 1_H)(m_i)$ and hence
$\mathcal M$ is finitely generated as a $(\mathcal C\# H)$-module.

\smallskip
Conversely, let $\mathcal{M}$ in $mod$-$(\mathcal{C}\#H)$ be finitely generated in $Mod$-$(\mathcal C\# H)$ and let 
\begin{equation}\label{4epx} \eta:V\otimes \left(\bigoplus_{i=1}^n {\bf{h}}_{X_i}\right)\longrightarrow \mathcal{M}
\end{equation} denote the epimorphism in $Mod$-$(\mathcal C\# H)$ as in Proposition \ref{fg}. In particular, $\eta$ is 
an epimorphism in $Mod$-$\mathcal C$. Then, if $\{v_1,...,v_k\}$ is a basis for $V$, it follows from the epimorphism in \eqref{4epx} that $\{m_{ij}=\eta(X_i)\big(v_j \otimes \text{id}_{X_i}\big) \in \mathcal{M}(X_i)\}_{1\leq i\leq n,1\leq j\leq k}$ gives a finite set of generators for $\mathcal{M}$ as a right $\mathcal{C}$-module.
\end{proof}

We remark here that if $V$ and $V'$ are left $H$-modules, then $Hom_K(V,V')$ carries a left $H$-module action defined by
\begin{equation}
(hf)(v)=\sum h_1f(S(h_2)v) \qquad \forall~v\in V, h\in H
\end{equation} This may be seen as  the special case of the action described in Proposition \ref{2.5} when $\mathcal C$ is the category with one object
having endomorphism ring $K$.  

\begin{lemma}\label{3.10}
Let $V\in H$-$Mod$ and $\mathcal{N}\in Mod$-$(\mathcal{C}\#H)$. Then, $Hom_K(V,\mathcal{N}(X))$ and $Hom_{Mod\text{-}\mathcal{C}}\big(V \otimes {\bf{h}}_{X}, \mathcal{N} \big)$ are isomorphic as objects in $H$-$Mod$ for each $X\in \text{Ob}(\mathcal{C})$.
\end{lemma}
\begin{proof}
We check that the canonical isomorphism $\phi: Hom_{Mod\text{-}\mathcal{C}}\big(V \otimes {\bf{h}}_{X}, \mathcal{N} \big) \longrightarrow Hom_K(V,Hom_{Mod\text{-}\mathcal C}({\bf h}_X,\mathcal N))\cong Hom_K(V,\mathcal{N}(X))$ defined by 
$\phi(\eta)(v):=\eta(X)( v \otimes \text{id}_X)$ for any morphism $\eta\in Hom_{Mod\text{-}\mathcal{C}}\big(V \otimes {\bf{h}}_{X}, \mathcal{N} \big)$ and $v \in V$ is $H$-linear:
\begin{align*}
\phi(h\eta)(v)&=(h\eta)(X)(v \otimes id_X)\\
              &= \sum h_1\eta(X)\big(S(h_2)(v \otimes id_X)\big)\\
              &= \sum h_1\eta(X)\big( S(h_3)v\otimes S(h_2)id_X\big)\\
             &= \sum h_1\eta(X)\big( S(h_3)v\otimes\varepsilon(S(h_2))id_X\big)\\
             &= \sum h_1\eta(X)\big( S(h_3)v\otimes\varepsilon(h_2)id_X\big)\\
             &=\sum h_1\eta(X)\big( S(h_2)v\otimes id_X\big)\\
              &= \sum h_1\phi(\eta)(S(h_2)v)=(h\phi(\eta))(v)
\end{align*}

\end{proof}

\begin{prop}\label{Mfg}
Let $\mathcal{M}$ and $\mathcal{N}$ be in $mod$-$(\mathcal{C}\#H)$ with $\mathcal{M}$ finitely generated in $Mod$-$(\mathcal{C}\#H)$. Then, $ Hom_{Mod\text{-}\mathcal{C}}(\mathcal{M},\mathcal{N})$ is $H$-locally finite, i.e., $\mathcal{L}_{Mod\text{-}\mathcal{C}}(\mathcal{M},\mathcal{N}) = Hom_{Mod\text{-}\mathcal{C}}(\mathcal{M},\mathcal{N})^{(H)}= Hom_{Mod\text{-}\mathcal{C}}(\mathcal{M},\mathcal{N})$.
\end{prop}
\begin{proof}
Since $\mathcal M\in mod$-$(\mathcal C\# H)$ is finitely generated in $Mod$-$(\mathcal C\# H)$, there exists by Proposition \ref{fg} a finite dimensional $H$-module $V$ and an epimorphism $\varphi:  V\otimes\left(\bigoplus_{i=1}^n {\bf{h}}_{X_i}\right)  \longrightarrow \mathcal{M}$ in $Mod$-$(\mathcal C\# H)$ for finitely many objects $X_1,\ldots,X_n$ in $\mathcal{C}$. Thus we get a monomorphism
\begin{equation}\label{mono4.3}\hat{\varphi}:Hom_{Mod\text{-}\mathcal{C}}(\mathcal{M},\mathcal{N}) \longrightarrow Hom_{Mod\text{-}\mathcal{C}}\left( V\otimes \left(\bigoplus_{i=1}^n  {\bf{h}}_{X_i}\right) , \mathcal{N}\right), \quad
\eta \mapsto \eta \circ \varphi.
\end{equation}
For each $X \in \text{Ob}(\mathcal{C})$, $v \in V$ and $f \in {\bf{h}}_{X_i}(X)$ for some chosen $1\leq i\leq n$, we have
\begin{align*}
\hat{\varphi}(h\eta)(X)(v \otimes f)&=(h\eta \circ \varphi)(X)(v \otimes f)
									=(h\eta)(X)\varphi(X)(v \otimes f)
									=\sum h_1 \eta(X)\big(S(h_2)\varphi(X)(v \otimes f)\big)\\
									&=\sum h_1 \eta(X)\big(\varphi(X)(S(h_2)(v \otimes f))\big)
									=\sum h_1 (\eta \circ \varphi)(X)(S(h_2)(v \otimes f))\\
									&=(h(\eta \circ \varphi))(X)(v \otimes f)
									=(h\hat{\varphi}(\eta))(X)(v \otimes f)
									\end{align*}
This shows that $\hat{\varphi}$ is an $H$-module monomorphism. By Lemma \ref{3.10}, we know that  
$$Hom_{Mod\text{-}\mathcal{C}}\left(V\otimes \left(\bigoplus_{i=1}^n {\bf{h}}_{X_i}\right), \mathcal{N}\right)\cong \bigoplus_{i=1}^nHom_K(V,\mathcal{N}(X_i))$$		
as $H$-modules. Since $V$ is finite dimensional, we know that $Hom_K(V,\mathcal{N}(X_i))\cong 
\mathcal N(X_i)\otimes V^*$ in $Vect_K$ and it is easily seen that this is an isomorphism of $H$-modules. Since  $\mathcal{N}$ is an $H$-locally finite $(\mathcal{C}\#H)$-module and $V^*$ is $H$-locally finite (because $dim_K(V^*)<\infty$),  it follows that each 
$\mathcal N(X_i)\otimes V^*$   is $H$-locally finite. The embedding in \eqref{mono4.3} now  shows that $Hom_{Mod\text{-}\mathcal{C}}(\mathcal{M},\mathcal{N})$ is $H$-locally finite. 					
\end{proof}

\begin{lemma}\label{hlocfinx} Let $\mathcal M$, $\mathcal N\in Mod$-$(\mathcal C\# H)$. For a morphism $\eta\in Hom_{Mod\text{-}
\mathcal C}(\mathcal M,\mathcal N)$, the following are equivalent:

\smallskip
(1) $\eta\in Hom_{Mod\text{-}\mathcal C}(\mathcal M,\mathcal N)^{(H)}=\mathcal L_{Mod\text{-}\mathcal C}(\mathcal M,\mathcal N)$. 

\smallskip
(2) There exists a finite dimensional $H$-module $V$, an element $v\in V$  and some $\hat\eta\in Hom_{Mod\text{-}(\mathcal C
\# H)}(V\otimes 
\mathcal M,\mathcal N)$ such that $\hat\eta(X)(v\otimes m)=\eta(X)(m)$ for each $X\in Ob(\mathcal C)$ and $m\in \mathcal M(X)$.  
\end{lemma}

\begin{proof}
(1) $\Rightarrow$ (2) : We put $V=H\eta$. Since $\eta\in  Hom_{Mod\text{-}\mathcal C}(\mathcal M,\mathcal N)^{(H)}$, we see that $V$ is finite dimensional. Let $\{\eta_1,...,\eta_k\}$ be a basis for $V=H\eta$. Any element $h\eta\in V$ can now be expressed
as $h\eta=\sum_{i=1}^k\alpha_i(h)\eta_i$. 

\smallskip We now define $\hat\eta\in Hom_{Mod\text{-}\mathcal C}(V\otimes\mathcal M,\mathcal N)$ by setting 
$
 \hat\eta(X)(h\eta\otimes m):=(h\eta)(X)(m)
$ for each $h\in H$, $X\in Ob(\mathcal C)$ and $m\in \mathcal M(X)$. It is clear that $\hat\eta(X)(\eta\otimes m)=\eta(X)(m)
$.

\smallskip
In order to show that $\hat\eta\in Hom_{Mod\text{-}(\mathcal C
\# H)}(V\otimes 
\mathcal M,\mathcal N)$, it suffices to show that each $\hat\eta(X):V\otimes \mathcal M(X)\longrightarrow\mathcal N(X)$
is $H$-linear. For $h'\in H$, we have
\begin{equation*}
\begin{array}{ll}
\hat\eta(X)(h'(h\eta\otimes m))&=\sum \hat\eta(X)( h'_1h\eta\otimes h'_2m )=\sum \sum_{i=1}^k \hat\eta(X)(\alpha_i(h'_1h)\eta_i\otimes h_2'm)\\&=\sum \sum_{i=1}^k\alpha_i(h'_1h)\hat\eta(X)(\eta_i\otimes h_2'm)=\sum \sum_{i=1}^k\alpha_i(h_1'h)\eta_i(X)(h_2'm)\\
&=\sum (h'_1h\eta)(X)(h'_2m)=\sum h'_1h_1\eta(X)(S(h_2)S(h'_2)h_3'm)\\
&= \sum h_1'h_1\eta(X)(S(h_2)\varepsilon(h_2')m)\\
&= \sum h'h_1\eta(X)(S(h_2)m)=\sum h'((h\eta)(X)(m))=h'\hat\eta(X)(h\eta\otimes m)
\end{array}
\end{equation*}

\smallskip
(2) $\Rightarrow$ (1) : We are given $\hat\eta\in Hom_{Mod\text{-}(\mathcal C
\# H)}(V\otimes 
\mathcal M,\mathcal N)$. Let $\{v_1,...,v_k\}$ be a basis for $V$ and suppose that $hv=\sum_{i=1}^k \alpha_i(h)v_i$. For each $1\leq i\leq k$, we define $\xi_i\in Hom_{Mod\text{-}\mathcal C}(\mathcal M,
\mathcal N)$ by setting $\xi_i(X)(m):=\hat\eta(X)(v_i\otimes m)$ for $X\in Ob(\mathcal C)$, $m\in \mathcal M(X)$. For any $h\in H$, we see that
\begin{equation*}
\begin{array}{ll}
(h\eta)(X)(m)&= \sum h_1\eta(X)(S(h_2)m) =\sum h_1\hat\eta(X)(v\otimes S(h_2)m)\\ 
&= \sum \hat\eta(X)(h_1v\otimes h_2S(h_3)m)=\hat\eta(X)(hv\otimes m)\\
&= \sum_{i=1}^k\alpha_i(h)\hat\eta(X)(v_i\otimes m) =\sum_{i=1}^k\alpha_i(h)\xi_i(X)(m)\\
\end{array}
\end{equation*}  It follows from the above that  $H\eta$ lies in the space generated by the finite collection
$\{\xi_1,...,\xi_k\}\in Hom_{Mod\text{-}\mathcal C}(\mathcal M,
\mathcal N)$. This proves the result. 
\end{proof}

\begin{prop}\label{exact}
If $\mathcal{I}$ is an injective object in $mod$-$(\mathcal{C}\#H)$, then $\mathcal{L}_{Mod\text{-}\mathcal{C}}\big({-},\mathcal{I}\big)$ is an exact functor from $mod$-$(\mathcal{C}\#H)$ to $H$-$mod$.
\end{prop}
\begin{proof}
Let $0\longrightarrow \mathcal{M} \overset{i}{\longrightarrow} \mathcal{N} \longrightarrow \mathcal{P} \longrightarrow 0$ be an exact sequence in $mod$-$(\mathcal{C}\#H)$ and $\mathcal{I}$ be an injective object in $mod$-$(\mathcal{C}\#H)$. Then,  $0\longrightarrow Hom_{Mod\text{-}\mathcal{C}}(\mathcal{P},\mathcal{I}) \longrightarrow Hom_{Mod\text{-}\mathcal{C}}(\mathcal{N},\mathcal{I}) \longrightarrow Hom_{Mod\text{-}\mathcal{C}}(\mathcal{M},\mathcal{I})$ is an exact sequence in $H$-$Mod$. Since the functor $(-)^{(H)}: H$-$Mod\longrightarrow H$-$mod$ is a right adjoint by Proposition \ref{adj}(1), it preserves monomorphisms. Thus, $0\longrightarrow \mathcal{L}_{Mod\text{-}\mathcal{C}}(\mathcal{P},\mathcal{I}) \longrightarrow \mathcal{L}_{Mod\text{-}\mathcal{C}}(\mathcal{N},\mathcal{I}) \longrightarrow \mathcal{L}_{Mod\text{-}\mathcal{C}}(\mathcal{M},\mathcal{I})$ is an exact sequence in $H$-$mod$. 

\smallskip Let $\eta \in  \mathcal{L}_{Mod\text{-}\mathcal{C}}(\mathcal{M},\mathcal{I})$. We set $V:=H\eta$. Then, $V$ is a finite dimensional $H$-module and therefore $V\in H$-$mod$. Thus, $ V \otimes \mathcal{M},  V \otimes \mathcal{N} \in mod$-$(\mathcal{C}\# H)$ and we have a monomorphism ${id}_V \otimes i: V \otimes \mathcal{M} \longrightarrow V \otimes \mathcal{N}$ in $mod$-$(\mathcal{C}\# H)$. Now, we consider the morphism $\zeta\in Hom_{Mod\text{-}\mathcal C
}(V\otimes \mathcal{M},\mathcal{I})$ defined by setting $\zeta(X)(\nu \otimes m):=\nu(X)(m)$ for each $X \in Ob(\mathcal{C})$, $m \in \mathcal{M}(X)$ and $\nu \in V$. It may be verified that $\zeta(X)$ is $H$-linear for each $X\in Ob(\mathcal{C})$. Thus, $\zeta\in Hom_{mod\text{-}(\mathcal C
\# H)}(V\otimes \mathcal{M},\mathcal{I})$. Since $\mathcal{I}$ is injective in $mod$-$(\mathcal{C}\#H)$, there exists a morphism $\xi:V \otimes\mathcal{N} \longrightarrow \mathcal I$ in $mod$-$(\mathcal{C}\#H)$ such that $\xi(\text{id}_V\otimes i)=\zeta$. The morphism $\xi\in Hom_{Mod\text{-}(\mathcal C
\# H)}(V\otimes 
\mathcal N,\mathcal I)$ now induces a morphism $\hat{\xi}\in Hom_{Mod\text{-}
\mathcal C}(\mathcal N,\mathcal I)$ defined by setting $\hat{\xi}(X)(n):= \xi(X)(\eta\otimes n)$ for every  $X\in Ob(\mathcal{C})$ and $n\in \mathcal{N}(X)$. Applying Lemma \ref{hlocfinx}, we see that $\hat{\xi}\in\mathcal{L}_{Mod\text{-}\mathcal{C}}(\mathcal{N},\mathcal{I})$. Also, $\hat{\xi} \circ i=\eta$. This completes the proof.
\end{proof}

\begin{prop}\label{freeres}
Let $\mathcal{C}$ be a left $H$-locally finite category which is  right noetherian. Let $\mathcal{M} \in mod\text{-}(\mathcal{C}\#H)$ be finitely generated as an object in $Mod$-$(\mathcal{C}\#H)$. If $\mathcal{I}$ is an injective object in $mod$-$(\mathcal{C}\#H)$, then $\textnormal{Ext}_{Mod\text{-}\mathcal{C}}^p(\mathcal{M},\mathcal{I})= 0$ for all $p>0$.
\end{prop}
\begin{proof}
Since $\mathcal{M}\in mod$-$(\mathcal{C}\#H)$ is finitely generated in $Mod$-$(\mathcal{C}\#H)$, by Proposition \ref{fg}, there exists a finite dimensional $H$-module $V_0$ and an epimorphism $$\eta_{0}: \mathcal{P}_0:=V_0\otimes \left(\bigoplus_{i=1}^{n_0}  {\bf{h}}_{X_i}\right)  \longrightarrow \mathcal{M},$$ in $Mod$-$(\mathcal{C}\#H)$ for finitely many objects $\{X_i\}_{1\leq i\leq n_0}$ in $\mathcal{C}$,  where ${\bf{h}}_{X_i}$ are viewed as objects in $Mod$-$(\mathcal{C}\#H)$. Since $\mathcal{C}$ is $H$-locally finite, each ${\bf{h}}_{X_i}\in mod$-$(\mathcal{C}\#H)$.  Since $V_0$   is finite dimensional, we must have $V_0\in H$-$mod$. Thus, $\mathcal{P}_0 \in mod$-$(\mathcal{C}\#H)$. Using Proposition \ref{fg} and Corollary \ref{CHfgC}, it follows that $\mathcal{P}_0$ is finitely generated in $Mod$-$\mathcal{C}$. Since $\mathcal{C}$ is right noetherian, $\mathcal{P}_0$ is a noetherian right $\mathcal{C}$-module (see, for instance, \cite[$\S$ 3]{Mit1}). Since the submodule of a finitely generated noetherian module is finitely generated, the $(\mathcal{C}\# H)$-submodule $\mathcal{K}:= Ker(\eta_0)$ of $\mathcal{P}_0$ is finitely generated in $Mod$-$\mathcal{C}$. So, again using Proposition \ref{fg} and Corollary \ref{CHfgC} it follows that there exists a finite dimensional $H$-module $V_1$ and an epimorphism 
$$\eta_{1}: \mathcal{P}_1:=V_1\otimes \left(\bigoplus_{j=1}^{n_1} {\bf{h}}_{Y_j}\right) \longrightarrow \mathcal{K},$$ in $Mod$-$(\mathcal{C}\#H)$ for finitely many objects $\{Y_j\}_{1\leq j\leq n_1}$ in $\mathcal{C}$. Since $V_0$ and $V_1$ are finite dimensional  $K$-vector spaces, clearly $\mathcal{P}_0$ and $\mathcal{P}_1$ are free right $\mathcal{C}$-modules. Moreover, $Im(\eta_1) = \mathcal{K}= Ker(\eta_0)$. Thus, continuing in this way, we can construct a free resolution of the module $\mathcal{M}$ in the category $Mod$-$\mathcal C$:
$$\mathcal{P}_*=\hdots \longrightarrow \mathcal{P}_i\longrightarrow\hdots \longrightarrow \mathcal{P}_1\longrightarrow \mathcal{P}_0\longrightarrow \mathcal{M}\longrightarrow 0$$ Hence, we have 
$$ \textnormal{Ext}_{Mod\text{-}\mathcal C}^p(\mathcal{M},\mathcal{I})= H^p(Hom_{Mod\text{-}\mathcal C}(\mathcal{P}_*,\mathcal{I})),\ \ \ \ \ \ \ \ \ p\geq0$$
Since $\mathcal M$  and $\{\mathcal P_i\}_{i\geq 0}$ are finitely generated in $Mod\text{-}(\mathcal{C}\#H)$, we have $\mathcal{L}_{Mod\text{-}\mathcal C}(\mathcal{M},\mathcal{I}) = Hom_{Mod\text{-}\mathcal C}(\mathcal{M},\mathcal{I})$ and $\mathcal{L}_{Mod\text{-}\mathcal C}(\mathcal{P}_i,\mathcal{I}) = Hom_{Mod\text{-}\mathcal C}(\mathcal{P}_i,\mathcal{I})$ for every $i\geq0$, by Proposition \ref{Mfg}. From Proposition \ref{exact}, we know that $\mathcal{L}_{Mod\text{-}\mathcal C}(
-,\mathcal I)$ is exact and it follows that $H^p(\mathcal{L}_{Mod\text{-}\mathcal C}(\mathcal{P}_*,\mathcal{I}))=0$ for all $p>0$. This proves the result.
\end{proof}

\begin{prop}\label{3.14}
Let $\mathcal{C}$ be a left $H$-locally finite category which is  right noetherian.  Let $\mathcal{M}$ and $\mathcal{N}$ be in $mod$-$(\mathcal{C}\#H)$ with $\mathcal{M}$ finitely generated in $Mod$-$(\mathcal{C}\#H)$ and let $\mathcal{E}^*$ be an injective resolution of $\mathcal{N}$ in $mod$-$(\mathcal{C}\#H)$. Then,
$$\textnormal{Ext}_{Mod\text{-}\mathcal C}^p(\mathcal{M},\mathcal{N})=H^p\big(Hom_{Mod\text{-}\mathcal C}(\mathcal{M},\mathcal{E}^*)\big), \quad p \geq 0.$$
\end{prop}
\begin{proof}
Let $\mathcal{P}_*$ be the free resolution of $\mathcal{M}$ in $Mod$-$\mathcal{C}$ constructed as in the proof of Proposition \ref{freeres}. Then, we have
\begin{equation*}
\textnormal{Ext}_{Mod\text{-}\mathcal{C}}^p(\mathcal{M},\mathcal{N})=H^p\big(Hom_{Mod\text{-}\mathcal{C}}(\mathcal{P}_*,\mathcal{N})\big)=H^p\big(\mathcal{L}_{Mod\text{-}\mathcal{C}}(\mathcal{P}_*,\mathcal{N})\big)
\end{equation*}
where the second equality follows from Proposition \ref{Mfg}. Since $H$-$mod$ is an abelian category and $\mathcal{L}_{Mod\text{-}\mathcal{C}}(\mathcal{P}_*,\mathcal{N})$ is a complex in $H$-$mod$,  it follows that $H^p\big(\mathcal{L}_{Mod\text{-}\mathcal{C}}(\mathcal{P}_*,\mathcal{N})\big)\in H$-$mod$. Hence, we may consider the family $\{\textnormal{Ext}_{Mod\text{-}\mathcal{C}}^p(\mathcal{M},-)\}_{p\geq 0}$ as  a $\delta$-functor from $mod$-$(\mathcal C\# H)$ to $H$-$mod$.

\smallskip
  By Proposition \ref{freeres}, $\textnormal{Ext}_{Mod\text{-}\mathcal{C}}^p(\mathcal{M},\mathcal{I})=0, ~p>0$ for every injective object $\mathcal{I}$ in $mod\text{-}(\mathcal{C} \#H)$. Since $mod$-$(\mathcal C\# H)$ has enough injectives, it follows
  that each $\textnormal{Ext}_{Mod\text{-}\mathcal{C}}^p(\mathcal{M},-):mod$-$(\mathcal C\# H)\longrightarrow H$-$mod$ is effaceable (see, for instance, \cite[$\S$ III.1]{Hart}). 
  
  \smallskip Since $mod$-$(\mathcal C\# H)$ has enough injectives, we can consider the right derived functors
  \begin{equation*} R^p\mathcal{L}_{Mod\text{-}\mathcal{C}}(\mathcal{M},-):mod\text{-}(\mathcal{C} \#H)\longrightarrow H\text{-}mod
  \end{equation*} For $p=0$, we notice that
$\textnormal{Ext}_{Mod\text{-}\mathcal{C}}^0(\mathcal{M},-)=Hom_{Mod\text{-}\mathcal{C}}(\mathcal{M},-)=\mathcal{L}_{Mod\text{-}\mathcal{C}}(\mathcal{M},-)=R^0\mathcal{L}_{Mod\text{-}\mathcal{C}}(\mathcal{M},-)$ as functors from $mod\text{-}(\mathcal{C} \#H)$ to $H$-$mod$. Since each $\textnormal{Ext}_{Mod\text{-}\mathcal{C}}^p(\mathcal{M},-)$ is effaceable for $p>0$, the family $\{\textnormal{Ext}_{Mod\text{-}\mathcal{C}}^p(\mathcal{M},-)\}_{p\geq 0}$ forms a universal $\delta$-functor and it follows from \cite[Corollary III.1.4]{Hart} that
\begin{equation*}\textnormal{Ext}_{Mod\text{-}\mathcal{C}}^p(\mathcal{M},-)=R^p\mathcal{L}_{Mod\text{-}\mathcal{C}}(\mathcal{M},-):mod\text{-}(\mathcal{C} \#H)\longrightarrow H\text{-}mod
\end{equation*} for every $p\geq 0$. Therefore, we have
$$\textnormal{Ext}_{Mod\text{-}\mathcal{C}}^p(\mathcal{M},\mathcal{N})=(R^p\mathcal{L}_{Mod\text{-}\mathcal{C}}(\mathcal{M},-))(\mathcal{N})=H^p(\mathcal{L}_{Mod\text{-}\mathcal{C}}(\mathcal{M},\mathcal{E}^*))=H^p\big(Hom_{Mod\text{-}\mathcal{C}}(\mathcal{M},\mathcal{E}^*)\big)$$ 

\end{proof}

\begin{theorem}\label{Tb4.18} Let $\mathcal{C}$ be a left $H$-locally finite category which is  right noetherian. Fix $\mathcal{M}\in mod$-$(\mathcal{C}\#H)$ with $\mathcal{M}$ finitely generated in $Mod$-$(\mathcal{C}\#H)$. We consider the functors
\begin{equation*}
\begin{array}{c} \mathscr F=Hom_{Mod\text{-}\mathcal{C}}(\mathcal{M},-):mod\text{-}(\mathcal{C}\#H)\longrightarrow H\text{-}mod
\qquad  \mathcal{N}\mapsto Hom_{Mod\text{-}\mathcal{C}}(\mathcal{M},\mathcal{N})\\
\mathscr G=(-)^H:H\text{-}mod\longrightarrow Vect_K\qquad M\mapsto M^H\\
\end{array}
\end{equation*} Then, we have the following spectral sequence
$$R^{p}(-)^H\big(\textnormal{Ext}^q_{Mod\text{-}\mathcal{C}}(\mathcal{M},\mathcal{N})\big)\Rightarrow \left(R^{p+q}Hom_{mod\text{-}(\mathcal{C}\#H)}(\mathcal{M},-)\right)(\mathcal{N})$$
\end{theorem}
\begin{proof}
 Using Proposition \ref{2.5} and the fact that $mod$-$(\mathcal{C}\#H)$ is a full subcategory of $Mod$-$(\mathcal{C}\#H)$, we have $(\mathscr{G}\circ\mathscr{F})(\mathcal{N})=Hom_{Mod\text{-}\mathcal{C}}(\mathcal{M},\mathcal{N})^H=Hom_{mod\text{-}(\mathcal{C}\#H)}(\mathcal{M},\mathcal{N})$. By definition, 
 \begin{equation*}
 R^q\mathscr F(\mathcal N)=H^q(\mathscr F(\mathcal E^*))=H^q(Hom_{Mod\text{-}\mathcal C}(\mathcal M,\mathcal E^*))
 \end{equation*} where   $\mathcal{E}^*$ is an injective resolution of $\mathcal{N}$ in $mod$-$(\mathcal{C}\#H)$. By   Proposition \ref{3.14}, we get $ R^q\mathscr F(\mathcal N)=\textnormal{Ext}^q_{Mod\text{-}\mathcal C}(\mathcal M,\mathcal N)$.  For any injective $\mathcal{I}$ in $mod$-$(\mathcal{C}\# H)$, we know that $\mathscr{F}(\mathcal{I})$ is an injective in $H$-$mod$ by Proposition \ref{Mfg} and   Lemma \ref{inj}(2). Since the category $H\text{-}mod$ has enough injectives (see,\cite[Lemma 1.4]{Gue}), the result now follows from Grothendieck spectral sequence for composite functors (see \cite{Grothen}).
\end{proof}

\section{Cohomology of relative $(\mathcal D,H)$-Hopf modules}

 Let $\mathcal{D}$ be a right co-$H$-category. In the notation of Definition \ref{coHcat}, for any $X$, $Y\in
 Ob(\mathcal D)$, there is an $H$-coaction  on the $K$-vector space $Hom_{\mathcal{D}}(X,Y)$ given  by $\rho_{XY}(f):=\sum f_0 \otimes f_1$. In this section, we will study the relative Hopf modules over the category $\mathcal{D}$ and describe their
 derived $Hom$-functors by means of spectral sequences. 

\smallskip
We denote by $Comod$-$H$ the category   of right $H$-comodules. If $M$ is an $H$-comodule with right $H$-coaction given by
$\rho_M:M\longrightarrow M\otimes H$, we set $M^{coH}:=\{m\in M~|~ \rho_M(m)=m \otimes 1_H\}$ to be the coinvariants
of $M$. 

\begin{definition}\label{Dr5.1}
Let $\mathcal{D}$ be a right co-$H$-category. Let $\mathcal{M}$ be a left $\mathcal{D}$-module with a given right $H$-comodule structure $\rho_{\mathcal{M}(X)}:\mathcal{M}(X) \longrightarrow \mathcal{M}(X) \otimes H$, $m \mapsto \sum m_0 \otimes m_1$ on
$\mathcal M(X)$  for each $X$ in $Ob(\mathcal{D})$. Then, $\mathcal{M}$ is said to be a relative $(\mathcal{D},H)$-Hopf module if the following condition holds:
\begin{equation}\label{relative}
\rho_{\mathcal{M}(Y)}\big(\mathcal{M}(f)(m)\big) = \sum \mathcal{M}(f_0)(m_0)\otimes f_1m_1
\end{equation}
for any $f\in Hom_{\mathcal{D}}(X,Y)$ and $m\in \mathcal{M}(X)$.
%where $\rho_{\mathcal{M}(X)}(m)=\sum m_0 \otimes m_1,$ $\rho_{XY}(f)=\sum f_0 \otimes f_1$ as in Definition \ref{coHcat}.
We denote by $_{\mathcal{D}}{\mathscr{M}}^H$ the category whose objects are relative $(\mathcal{D},H)$-Hopf modules and whose morphisms are given by
$$Hom_{_{\mathcal{D}}{\mathscr{M}}^H}(\mathcal{M},\mathcal{N}):=\{\eta \in Hom_{\mathcal{D}\text{-}Mod}(\mathcal{M},\mathcal{N})\ |\ \eta(X):\mathcal{M}(X) \longrightarrow \mathcal{N}(X)~ \text{is $H$-colinear}~ \forall X \in Ob(\mathcal{D}) \}.$$
\end{definition}
 
We now recall the tensor product of $H$-comodules. Let $M,N \in Comod$-$H$ with $H$-coactions $\rho_M$ and $\rho_{N}$, respectively. Then, $M\otimes N \in Comod$-$H$ with $H$-coaction given by $\rho_{M \otimes N}:=(id \otimes id \otimes m_H)(id\otimes \tau \otimes id)(\rho_M\otimes\rho_N)$, where $m_H$ denotes the multiplication on $H$ and $id\otimes \tau\otimes id: M\otimes (H\otimes N)\otimes H\longrightarrow M\otimes (N\otimes H)\otimes H $ denotes the twist map. In other words, $\rho_{M \otimes N}(m\otimes n)= \sum m_0\otimes n_0\otimes m_1n_1$ for $m \otimes n \in M \otimes N$.

\begin{lemma}\label{tenprod}
Let $M \in Comod\text{-}H$ and $\mathcal{N}\in  {_{\mathcal{D}}}{\mathscr{M}}^H$. Then, $\mathcal{N}\otimes M$  defined by setting  
\begin{equation*}
\begin{array}{c}
(\mathcal{N}\otimes M)(X) := \mathcal{N}(X) \otimes M\\
(\mathcal{N}\otimes M)(f)(n \otimes m) := \mathcal{N}(f)(n)\otimes m
\end{array}
\end{equation*}
for each $X \in \text{Ob}(\mathcal{D}),~ f \in Hom_{\mathcal{D}}(X,Y)$ and $n \otimes m \in \mathcal{N}(X) \otimes M$ is a relative
$(\mathcal D,H)$-Hopf module. 
\end{lemma}

\begin{proof}
Clearly, $\mathcal{N}\otimes M$ is a left $\mathcal{D}$-module. Since $\mathcal{N}(X)$ is a right $H$-comodule, $\mathcal{N}(X) \otimes M$ also carries  a right $H$-comodule structure for each $X \in \text{Ob}(\mathcal{D})$. For any $ f\in Hom_{\mathcal{D}}(X,Y)$ and $n \otimes m \in \mathcal{N}(X) \otimes M$, we have
\begin{align*}
\rho\big((\mathcal{N} \otimes M)(f)(n \otimes m)\big) &= \rho\big(\mathcal{N}(f)(n)\otimes m \big)\\
&= \sum{(\mathcal{N}(f)(n))}_0 \otimes m_0 \otimes {(\mathcal{N}(f)(n))}_1m_1\\
&= \sum\mathcal{N}(f_0)(n_0) \otimes m_0 \otimes f_1n_1m_1\\
&= \sum(\mathcal{N} \otimes M)(f_0)(n_0 \otimes m_0)  \otimes f_1n_1m_1\\
&= \sum(\mathcal{N} \otimes M)(f_0)(n \otimes m)_0 \otimes f_1(n \otimes m)_1 
\end{align*}
This shows that $\mathcal N\otimes M$ satisfies the condition  \eqref{relative} in Definition \ref{Dr5.1}. 
\end{proof}

From Lemma \ref{tenprod}, it follows that the assignment $M \mapsto \mathcal{N} \otimes M$ defines a functor $\mathcal{N} \otimes (-):Comod\text{-}H \longrightarrow {_{\mathcal{D}}}{\mathscr{M}}^H$ for each $\mathcal{N} \in {_{\mathcal{D}}}{\mathscr{M}}^H$.

\smallskip
From the definition of a co-$H$-category, it is also clear that the $\mathcal{D}$-module ${_{X}}{\bf{h}}:=Hom_{\mathcal{D}}(X,{-}) $ lies in ${_{\mathcal{D}}}{\mathscr{M}}^H$ for each $X\in Ob(\mathcal{D})$. 

\begin{lemma}\label{gen1}
Let $\mathcal{M}$ be a relative $(\mathcal D,H)$-Hopf module and let $m\in \mathcal{M}(X)$ for some $X\in Ob(\mathcal{D})$. Then, there exists a finite dimensional $H$-comodule $W_m\subseteq \mathcal M(X)$ containing $m$ and a  morphism $
\eta_m:{_{X}}{\bf{h}} \otimes W_m \longrightarrow \mathcal{M}$ in  ${_{\mathcal{D}}}{\mathscr{M}}^H$ such that 
$\eta_m(X)(id_X\otimes m)=m$.
\end{lemma}
\begin{proof}
Using 
\cite[Theorem 2.1.7]{SCS}, we know that there exists a finite dimensional $H$-subcomodule $W_m$ of $\mathcal{M}(X)$ containing $m$. We consider the $\mathcal{D}$-module morphism $\eta_m:{_{X}}{\bf{h}} \otimes W_m \longrightarrow \mathcal{M}$ defined by 
\begin{equation*}
\eta_m(Y)(f\otimes w):= \mathcal{M}(f)(w) 
\end{equation*}
for any $Y\in Ob(\mathcal{D})$, $f\in {_X}{\bf{h}}(Y)$ and $w\in W_m$. We now verify that $\eta_m$ is indeed a morphism in ${_{\mathcal{D}}}{\mathscr{M}}^H$, i.e., $\eta_m(Y)$ is $H$-colinear for each $Y\in Ob(\mathcal{D})$:
\begin{equation*}
\rho\big(\eta_m(Y)(f\otimes w)\big)=\rho\big(\mathcal{M}(f)(w)\big)=\sum\mathcal{M}(f_0)(w_0)\otimes f_1w_1=\sum \eta_m(Y)(f_0\otimes w_0)\otimes f_1w_1=\eta_m(Y)\left({(f\otimes w)}_0\right) \otimes {(f\otimes w)}_1
\end{equation*}

\end{proof}

\begin{remark}
\emph{It might be tempting to view Lemma \ref{gen1} as a Yoneda correspondence.  But, we note that the finite dimensional $H$-comodule and the morphism in Lemma \ref{gen1} determined by $m\in \mathcal{M}(X)$ need not be unique. }
\end{remark}

Given a morphism $\eta:\mathcal M\longrightarrow \mathcal N$ in ${_{\mathcal{D}}}{\mathscr{M}}^H$, it may be easily verified
that $Ker(\eta)$ and $Coker(\eta)$ determined by setting
\begin{equation}\label{epimonoy}
\begin{array}{c}
Ker(\eta)(X):=Ker(\eta(X):\mathcal M(X)\longrightarrow\mathcal N(X))\\
Coker(\eta)(X):=Coker(\eta(X):\mathcal M(X)\longrightarrow\mathcal N(X))\\
\end{array}
\end{equation} for each $X\in Ob(\mathcal D)$ are also relative $(\mathcal D,H)$-Hopf modules. It follows that  $\eta:\mathcal M\longrightarrow \mathcal N$ in ${_{\mathcal{D}}}{\mathscr{M}}^H$ is a monomorphism (resp. an epimorphism) if and only if it induces 
monomorphisms (resp. epimorphisms) $\eta(X):
\mathcal M(X)\longrightarrow\mathcal N(X)$ of $H$-comodules for each $X\in Ob(\mathcal D)$. 

\begin{prop}\label{fg(D,H)} 
Let $\mathcal{D}$ be a right co-$H$-category. Then, a module $\mathcal{M}\in {_{\mathcal{D}}}{\mathscr{M}}^H$ is finitely generated as an object in $\mathcal D\text{-}Mod$ if and only if there exists a finite dimensional $H$-comodule $W$ and an epimorphism 
\begin{equation*}
\left(\bigoplus_{i\in I} {_{X_i}}{\bf{h}} \right)\otimes W \longrightarrow \mathcal{M}
\end{equation*}
in  ${_{\mathcal{D}}}{\mathscr{M}}^H$, for finitely many objects $\{X_i\}_{i\in I}$  in $\mathcal{D}$.
%  where ${\bf{h}}_{X_i}:=Hom_{\mathcal{D}}(X_i,{-})\in  {_{\mathcal{D}}}{\mathscr{M}}^H$.
\end{prop}

\begin{proof}
Let $\mathcal{M}\in {_{\mathcal{D}}}{\mathscr{M}}^H$ be finitely generated as a $\mathcal D$-module. Then, there exists a finite collection $\{m_i\in el(\mathcal M)\}_{i\in I}$  such that every $y \in el(\mathcal M)$  has the form $y=\sum_{i\in I} \mathcal{M}(f_i)(m_i)$ for some $f_i \in Hom_{\mathcal{D}}(|m_i|,|y|)$. Applying Lemma \ref{gen1}, we can obtain for each $m_i$ a finite dimensional $H$-subcomodule  $W_{m_i}\subseteq \mathcal{M}(|m_i|)$ containing $m_i$ and   a morphism $\eta_{m_i}:{_{|m_i|}}{\bf{h}} \otimes W_{m_i} \longrightarrow \mathcal{M}$ in ${_{\mathcal{D}}}{\mathscr{M}}^H$. Setting $W:=\bigoplus_{i\in I} W_{m_i}$, we have an epimorphism in ${_{\mathcal{D}}}{\mathscr{M}}^H$ determined by 
\begin{equation*}
\eta:\left(\bigoplus_{i\in I}{_{|m_i|}}{\bf{h}}\right) \otimes W \longrightarrow \mathcal{M},~~~~~~~\eta(Y)(\{f_i\}_{i\in I}\otimes \{w_{j}\}_{j\in I}):= \sum_{i\in I}\mathcal{M}(f_i)(w_i)
\end{equation*}
for each $Y \in Ob(\mathcal{D})$, $f_i \in {_{|m_i|}}{\bf{h}}(Y)$ and $w_i\in W_{m_i}$.

\smallskip
 Conversely, let $\{w_1,\ldots,w_k\}$ be a basis of a finite dimensional $H$-comodule $W$ and $\{X_i\}_{i\in I}$ be finitely many objects in $\mathcal{D}$ such that we have an epimorphism 
\begin{equation*}
\eta: \left(\bigoplus_{i\in I} {_{X_i}}{\bf{h}}\right)\otimes W  \longrightarrow \mathcal{M}
\end{equation*} 
in ${_{\mathcal{D}}}{\mathscr{M}}^H$. From the discussion above, it follows that $\eta(Y):\left(\bigoplus_{i\in I} {_{X_i}}{\bf{h}}(Y)\right)\otimes W\longrightarrow \mathcal M(Y)$ is an epimorphism in $Comod$-$H$ for each $Y\in Ob(\mathcal D)$.   Then,  the elements 
$\{m_{ij}:=\eta(X_i)\big(\text{id}_{X_i} \otimes w_j\big)\}_{i\in I,1\leq j\leq k} $ form a family of generators for $\mathcal{M}$ as
a $\mathcal D$-module.
\end{proof}

We will now show that $_{\mathcal{D}}{\mathscr{M}}^H$ is a Grothendieck category. This essentially follows from the fact that both $\mathcal{D}$-$Mod$ and $Comod$-$H$ are Grothendieck categories. We refer the reader, for instance, to \cite[Corollary 2.2.8]{SCS}, for a proof of $Comod$-$H$ being a Grothendieck category.

\begin{prop}
Let $\mathcal{D}$ be a right co-$H$-category.  Then, the category $_{\mathcal{D}}{\mathscr{M}}^H$ of 
relative $(\mathcal D,H)$-Hopf modules is a Grothendieck category.
\end{prop}

\begin{proof}
Since the categories $\mathcal{D}$-$Mod$ and $Comod$-$H$ have kernels, cokernels and coproducts (direct sums), so does the category  $_{\mathcal{D}}{\mathscr{M}}^H$. The remaining properties of an abelian category are inherited by $_{\mathcal{D}}{\mathscr{M}}^H$ from $\mathcal{D}$-$Mod$. Hence, $_{\mathcal{D}}{\mathscr{M}}^H$ is a cocomplete abelian category. Directs limits are  exact in $_{\mathcal{D}}{\mathscr{M}}^H$ which is also a property inherited from $\mathcal{D}$-$Mod$. We are now left to check that $_{\mathcal{D}}{\mathscr{M}}^H$ has a family of generators. For any $\mathcal{M}$ in ${_{\mathcal{D}}}{\mathscr{M}}^H$, it follows from Lemma \ref{gen1} that we can find an epimorphism 
$$\bigoplus_{m\in el(\mathcal{M})}\eta_m:\bigoplus_{m\in el(\mathcal{M})}{_{|m|}\bf{h}}\otimes W_m\longrightarrow \mathcal{M}$$
in $_{\mathcal{D}}{\mathscr{M}}^H$. Thus, the collection $\{_X{\bf h}\otimes W\}$, where $X$ ranges over all objects in $\mathcal{D}$ and $W$ ranges over all (isomorphism classes of) finite dimensional $H$-comodules, forms a generating family for ${_{\mathcal{D}}}{\mathscr{M}}^H$ (see, for instance, the proof of \cite[Proposition 1.9.1]{Grothen}).
\end{proof}

For $\mathcal{N} \in {_{\mathcal{D}}}{\mathscr{M}}^H$, we consider the functor $\mathcal{N}\otimes {(-)}:Comod\textnormal{-}H$ $\longrightarrow {_{\mathcal{D}}}{\mathscr{M}}^H$ given by 
$M\mapsto\mathcal{N}\otimes M$. We see that  $Comod\textnormal{-}H$ is a Grothendieck category and the functor $\mathcal{N}\otimes {(-)}$ preserves colimits. Therefore, by a classical result \cite[Proposition 8.3.27(iii)]{KS}, it has a right adjoint which we denote by $\mathscr{R}_{\mathcal{N}}: {_{\mathcal{D}}}{\mathscr{M}}^H\longrightarrow Comod\textnormal{-}H$. We then define
\begin{equation}\label{HOM}
HOM_{\mathcal{D}{\text{-}}Mod}(\mathcal{N},\mathcal{P}):=\mathscr{R}_{\mathcal{N}}(\mathcal{P})
\end{equation}
for any $\mathcal{P} \in {_{\mathcal{D}}}{\mathscr{M}}^H$. Thus, we have a natural isomorphism
\begin{equation}\label{iso}
Hom_{{_{\mathcal{D}}}{\mathscr{M}}^H}\big(\mathcal{N}\otimes M,\mathcal{P}\big) \xrightarrow{\cong} Hom_{Comod\text{-}H}\big(M, HOM_{\mathcal{D}\text{-}Mod}(\mathcal{N},\mathcal{P})\big)
\end{equation}
for $\mathcal{N},\mathcal{P} \in {_{\mathcal{D}}}{\mathscr{M}}^H$ and $M \in Comod$-$H$.  In particular, when $\mathcal{D}$ is a right co-$H$-category with a  single object, i.e., a right $H$-comodule algebra, then $\mathcal{N}$ and $\mathcal{P}$ are  relative Hopf-modules in the classical sense of Takeuchi \cite{t}. Then, using \cite[Lemma 2.3]{CanGue}, the definition of $HOM$ as in \eqref{HOM} recovers the standard definition of rational morphisms between relative Hopf modules as in  \cite[$\S$ 2]{CanGue} or \cite{Ulb}. As such, we will refer to $HOM_{\mathcal D\text{-}Mod}(-,-)$ as the  ``rational $Hom$
object'' in ${_{\mathcal{D}}}{\mathscr{M}}^H$. 

\begin{corollary}\label{H-coH}
 Let $\mathcal{N},\mathcal{P}\in{_{\mathcal{D}}}{\mathscr{M}}^H$. Then, $HOM_{\mathcal{D}\text{-}Mod}(\mathcal{N},\mathcal{P})^{coH}=Hom_{{_{\mathcal{D}}}{\mathscr{M}}^H}(\mathcal{N},\mathcal{P})$.
 \end{corollary}
\begin{proof}
The result follows by choosing $M=k$ in \eqref{iso} and the fact that $Hom_{Comod\text{-}H}(k,N)=N^{coH}$ for any $N \in Comod$-$H$.
\end{proof}

\begin{corollary}\label{4.8}
If $\mathcal{I}$ is an injective in ${_{\mathcal{D}}}{\mathscr{M}}^H$, then $HOM_{\mathcal{D}\text{-}Mod}(\mathcal{N},\mathcal{I})$ is an injective in $Comod$-$H$ for any $\mathcal{N}$ in ${_{\mathcal{D}}}{\mathscr{M}}^H$.
\end{corollary}
\begin{proof}
The fact that $HOM_{\mathcal{D}\text{-}Mod}(\mathcal{N},-):{_{\mathcal{D}}}{\mathscr{M}}^H\longrightarrow Comod\text{-}H$ 
preserves injectives follows from the fact that its left adjoint $\mathcal N\otimes (-):Comod\text{-}H
\longrightarrow {_{\mathcal{D}}}{\mathscr{M}}^H$ is an exact functor.
\end{proof}

At the level of higher derived functors, the result of Corollary \ref{H-coH} leads to the following spectral sequence. 

\begin{theorem}\label{Tf5.9}
Let $\mathcal{M}\in {_{\mathcal{D}}}{\mathscr{M}}^H$ be a relative $(\mathcal D,H)$-Hopf module. We consider the functors 
\begin{equation*}
\begin{array}{c}\mathscr F=HOM_{\mathcal{D} \text{-}Mod}(\mathcal{M},-): {_{\mathcal{D}}}{\mathscr{M}}^H\longrightarrow Comod\text{-}H\qquad \mathcal{N}\mapsto HOM_{\mathcal{D} \text{-}Mod}(\mathcal{M},\mathcal{N})\\
\mathscr G=({-})^{coH}: Comod{\text{-}}H \longrightarrow Vect_K\qquad M\mapsto M^{coH}\\
\end{array}
\end{equation*} Then, we have the following spectral sequence $$R^{p}(-)^{coH}(R^qHOM_{\mathcal{D}\text{-}Mod}(\mathcal{M},-)(\mathcal{N}))\Rightarrow \left(R^{p+q}Hom_{{_{\mathcal{D}}}{\mathscr{M}}^H}(\mathcal{M},-)\right)(\mathcal{N})$$
\end{theorem}
\begin{proof}
We have $(\mathscr{G}\circ\mathscr{F})(\mathcal{N})= HOM_{\mathcal{D} \text{-}Mod}(\mathcal{M},\mathcal{N})^{coH} = Hom_{{_{\mathcal{D}}}{\mathscr{M}}^H}(\mathcal{M},\mathcal{N})$ by Corollary \ref{H-coH}. By Corollary \ref{4.8}, 
the functor $\mathscr F$ preserves injectives.  Since $Comod$-$H$ has enough injectives, the result now follows from Grothendieck spectral sequence for
composite functors (see \cite{Grothen}). 
\end{proof}

Let $M,N$ be right $H$-comodules. Let $H^*$ be the linear dual of $H$. Then, the space $Hom_K(M,N)$ carries a left $H^*$-module structure given by
\begin{equation*}
(h^*f)(m):=\sum h^*\left(S^{-1}(m_1){(f(m_0))}_1\right)({f(m_0))}_0
\end{equation*}
for any $h^* \in H^*$, $f \in Hom_K(M,N)$ and $m \in M$. We now show that this $H^*$-action can be extended to relative $(\mathcal{D}\text{-}H)$-Hopf modules.

\begin{lemma}
Let $\mathcal{M},~\mathcal{N}\in {_{\mathcal{D}}}{\mathscr{M}}^H$. Then, $Hom_{\mathcal{D}\text{-}Mod}(\mathcal{M},\mathcal{N})$ is a left $H^*$-module.
\end{lemma}

\begin{proof}
For $h^* \in H^*$ and $\eta \in Hom_{\mathcal{D}\text{-}Mod}(\mathcal{M},\mathcal{N})$, we set
\begin{equation}\label{H^*-action}
(h^*\eta)(X)(m):=\sum h^*\Big(S^{-1}(m_1){(\eta(X)(m_0))}_1\Big){(\eta(X)(m_0))}_0.
\end{equation}
for all $X \in Ob(\mathcal{D})$ and $m \in \mathcal{M}(X)$. We first verify that $h^*\eta$ is indeed an element in $Hom_{\mathcal{D}\text{-}Mod}(\mathcal{M},\mathcal{N})$. For any $f \in Hom_\mathcal{D}(X,Y)$, we have
\begin{equation*}
\begin{array}{ll}
(h^*\eta)(Y)\mathcal{M}(f)(m)&= \sum h^* \Big(S^{-1}\big({(\mathcal{M}(f)(m))}_1\big) {\big(\eta(Y){\big((\mathcal{M}(f)(m))}_0 \big)\big)}_1 \Big) {\big(\eta(Y){\big((\mathcal{M}(f)(m))}_0 \big)\big)}_0\\
& = \sum h^* \Big(S^{-1}(f_1m_1) {\big(\eta(Y)\big((\mathcal{M}(f_0)(m_0) \big)\big)}_1  \Big) {\big(\eta(Y)\big((\mathcal{M}(f_0)(m_0) \big)\big)}_0 ~~~~~~~(\text{using}~ \eqref{relative})\\
& = \sum h^* \Big(S^{-1}(m_1)S^{-1}(f_1) {\big(\mathcal{N}(f_0)\big((\eta(X)(m_0) \big)\big)}_1  \Big) {\big(\mathcal{N}(f_0)\big((\eta(X)(m_0) \big)\big)}_0\\
& = \sum h^* \Big(S^{-1}(m_1)S^{-1}(f_1)(f_0)_1(\eta(X)(m_0))_1\Big)\mathcal{N}((f_0)_0)(\eta(X)(m_0))_0~~~~~(\text{using}~ \eqref{relative}) \\
& = \sum h^* \Big(S^{-1}(m_1)S^{-1}(f_2)f_1(\eta(X)(m_0))_1\Big)\mathcal{N}(f_0)(\eta(X)(m_0))_0\\
& = \mathcal{N}(f)(h^*\eta)(X)(m)
\end{array}
\end{equation*}
Next, we verify that $(h^*g^*)\eta=h^*(g^*\eta)$ and that $1_{H^*}\eta=\eta$, i.e., $\varepsilon \eta=\eta$ for all $h^*,g^* \in H^*$ and $\eta \in Hom_{\mathcal{D}\text{-}Mod}(\mathcal{M},\mathcal{N})$. The latter equality follows easily and further we see that 
\begin{equation*}
\begin{array}{ll}
(h^*(g^*\eta))(X)(m)&= \sum h^*\Big(S^{-1}(m_1)\big((g^*\eta)(X)(m_0)\big)_1 \Big)\big((g^*\eta)(X)(m_0)\big)_0\\
&= \sum h^*\Big(S^{-1}(m_1)\Big(g^*\big(S^{-1}({(m_0)}_1)\big(\eta(X)({(m_0)}_0)\big)\big)_1  \big(\eta(X)((m_0)_0)\big)\big)_0 \Big)_1 \Big)\\
&~~~~~~~ \Big(g^*\big(S^{-1}({(m_0)}_1)\big(\eta(X)({(m_0)}_0)\big)\big)_1  \big(\eta(X)({(m_0)}_0)\big)\big)_0 \Big)_0\\
&= \sum h^*\Big(S^{-1}(m_2){(\eta(X)(m_0))}_1\Big)g^*\Big(S^{-1}(m_1)(\eta(X)(m_0))_2\Big)(\eta(X)(m_0))_0\\
&= \sum (h^*g^*)(S^{-1}(m_1){(\eta(X)(m_0))}_1){(\eta(X)(m_0))}_0\\
&= \big((h^*g^*)\eta\big)(X)(m)
\end{array}
\end{equation*} 
for all $X \in Ob(\mathcal{D})$ and $m \in \mathcal{M}(X)$.
\end{proof}

\begin{lemma}
Let $\mathcal{M},~\mathcal{N}\in {_{\mathcal{D}}}{\mathscr{M}}^H$ and let $\eta \in Hom_{\mathcal{D}\text{-}Mod}(\mathcal{M},\mathcal{N})$. Then, there is a morphism $\rho(\eta) \in Hom_{\mathcal{D}\text{-}Mod}(\mathcal{M},\mathcal{N}\otimes H)$ determined by setting
\begin{equation}\label{coaction on eta}
\rho(\eta)(X)(m):= \sum\big(\eta(X)(m_0)\big)_0 \otimes S^{-1}(m_1)\big(\eta(X)(m_0)\big)_1
\end{equation}
for any $X \in Ob(\mathcal{D})$ and $m \in \mathcal{M}(X)$. 
\end{lemma}

\begin{proof}
Using \eqref{relative} and the fact that $\eta \in Hom_{\mathcal{D}\text{-}Mod}(\mathcal{M},\mathcal{N})$, we have
\begin{equation*}
\begin{array}{ll}
\rho(\eta)(Y)\big(\mathcal{M}(f)(m)\big) &=\sum\Big(\eta(Y)\big(\mathcal{M}(f_0)(m_0)\big)\Big)_0\otimes S^{-1}(f_1m_1)\Big(\eta(Y)\big(\mathcal{M}(f_0)(m_0)\big)\Big)_1\\
&=\sum\Big(\mathcal{N}(f_0)\big(\eta(X)(m_0)\big)\Big)_0\otimes S^{-1}(f_1m_1)\Big(\mathcal{N}(f_0)\big(\eta(X)(m_0)\big)\Big)_1\\
&=\sum\mathcal{N}(f_0)\big(\eta(X)(m_0)\big)_0\otimes S^{-1}(m_1)S^{-1}(f_2)f_1\big(\eta(X)(m_0)\big)_1\\
%&=\sum\varepsilon(f_1)\mathcal{N}(f_0)\big(\eta(X)(m_0)\big)_0\otimes S^{-1}(m_1)\big(\eta(X)(m_0)\big)_1\\
&=\sum\mathcal{N}(f)\big(\eta(X)(m_0)\big)_0\otimes S^{-1}(m_1)\big(\eta(X)(m_0)\big)_1\\
&=(\mathcal{N}(f) \otimes id_H)\rho(\eta)(X)
\end{array}
\end{equation*}
for any $f\in Hom_{\mathcal{D}}(X,Y)$ and $m \in \mathcal{M}(X)$.
\end{proof}

We now recall the notion of a rational left $H^*$-module (see, for instance, \cite{SCS}) which will be used in the next result.
Given a left $H^*$-module $M$, there is a morphism $\rho_M:M\longrightarrow Hom_K(H^*,M)$ corresponding to 
the canonical morphism $H^*\otimes M\longrightarrow M$. There is an obvious inclusion $M\otimes H\hookrightarrow Hom_K(H^*,M)$
given by $(m\otimes h)(h^*)=h^*(h)m$ for any $m\in M$, $h\in H$ and $h^*\in H^*$. 

\begin{definition}\label{Def5.12} (see \cite[Definition 2.2.2]{SCS}) A left $H^*$-module $M$ is said to be rational
if $\rho_M(M)\subseteq M\otimes H$, where $M\otimes H$ is viewed as a subspace of $Hom_K(H^*,M)$. The full subcategory
of rational $H^*$-modules will be denoted by $Rat(H^*\text{-}Mod)$ . 
\end{definition}

 If $M$ is a right $H$-comodule with $H$-coaction $m \mapsto \sum m_0 \otimes m_1$, then $M$ becomes a left $H^*$-module via the action $h^*m:=\sum h^*(m_1)m_0$ for $h^* \in H^*$ and $m \in M$. This determines a functor 
 \begin{equation*}
 Comod\text{-}H\longrightarrow H^*\text{-}Mod
 \end{equation*} It is well known (see \cite[Theorem 2.2.5]{SCS}) that  this functor defines an equivalence of categories between $Comod\text{-}H$ and the subcategory  $Rat(H^*\text{-}Mod)$ of $H^*\text{-}Mod$. 

\begin{prop}\label{4.12}
Let $\mathcal{M},~\mathcal{N}\in {_{\mathcal{D}}}{\mathscr{M}}^H$ and suppose that $\mathcal{M}$ is finitely generated as an object in $\mathcal{D}\text{-}Mod$. Then, $Hom_{\mathcal{D}\text{-}Mod}(\mathcal{M},\mathcal{N})$ is a right $H$-comodule. In particular, $HOM_{\mathcal{D}\text{-}Mod}(\mathcal{M},\mathcal{N})= Hom_{\mathcal{D}\text{-}Mod}(\mathcal{M},\mathcal{N})$.
\end{prop}

\begin{proof}
Since $\mathcal{M}$ is finitely generated in $\mathcal{D}\text{-}Mod$, by Proposition \ref{fg(D,H)}, there exists
a finite dimensional $H$-comodule $W$ and an epimorphism 
\begin{equation*}
\eta:\left(\bigoplus_{i\in I} {_{X_i}}{\bf{h}} \right)\otimes W \longrightarrow \mathcal{M}
\end{equation*}
in  ${_{\mathcal{D}}}{\mathscr{M}}^H$, for finitely many objects $\{X_i\}_{i\in I}$  in $\mathcal{D}$. From the description
of epimorphisms in $ {_{\mathcal{D}}}{\mathscr{M}}^H$ in \eqref{epimonoy}, we know that $\eta$ is also an epimorphism in $\mathcal D\text{-}Mod$. 
The map
\begin{equation*}
Hom(\eta, \mathcal{N}): Hom_{\mathcal{D}\text{-}Mod}(\mathcal{M},\mathcal{N})\hookrightarrow \bigoplus_{i\in I} Hom_{\mathcal{D}\text{-}Mod}({_{X_i}}{\bf h}\otimes W,~\mathcal{N})
\end{equation*}
is therefore a monomorphism  for each $\mathcal{N} \in {_{\mathcal{D}}}{\mathscr{M}}^H$. Using the fact that $\eta(Y)$ is $H$-colinear for each $Y\in Ob(\mathcal{D})$, we will now verify that the morphism $Hom(\eta, \mathcal{N})$ is $H^*$-linear. For any $h^* \in H^*$, $\xi \in Hom_{\mathcal{D}\text{-}Mod}(\mathcal{M},\mathcal{N})$, $Y \in Ob(\mathcal{D})$ and $\tilde{f} \otimes w \in \left(\bigoplus_{i \in I} {_{X_i}}{\bf h}(Y) \right)\otimes W$, we have
\begin{equation*}
\begin{array}{ll}
&\left(Hom(\eta,\mathcal{N})(h^*\xi)\right)(Y)(\tilde{f} \otimes w)\\
& \quad =((h^*\xi)\circ \eta)(Y)(\tilde{f} \otimes w)=(h^*\xi)(Y)\left(\eta(Y)(\tilde{f} \otimes w)\right)\\
& \quad = \sum h^*\Big(S^{-1}\left({(\eta(Y)(\tilde{f} \otimes w))}_1\right){\big(\xi(Y){\left((\eta(Y)(\tilde{f} \otimes w)\right)}_0\big)}_1\Big){\left(\xi(Y){\left((\eta(Y)(\tilde{f} \otimes w)\right)}_0\right)}_0\\
& \quad = \sum h^*\Big(S^{-1}\left({(\tilde{f} \otimes w)}_1\right){\big(\xi(Y)\left(\eta(Y)(\tilde{f} \otimes w)_0\right)\big)}_1\Big){\big(\xi(Y)\left(\eta(Y)(\tilde{f} \otimes w)_0\right)\big)}_0\\
& \quad = \sum h^*\Big(S^{-1}\left({(\tilde{f} \otimes w)}_1\right){\big((\xi \circ \eta)(Y)(\tilde{f} \otimes w)_0)\big)}_1\Big){\big((\xi \circ \eta)(Y)(\tilde{f} \otimes w)_0)\big)}_0\\
& \quad = \left(h^*(\xi \circ \eta)\right)(Y)(\tilde{f} \otimes w)\\
& \quad =\left(h^*Hom(\eta,\mathcal{N})(\xi)\right)(Y)(\tilde{f} \otimes w)
\end{array}
\end{equation*}
This shows that $Hom_{\mathcal{D}\text{-}Mod}(\mathcal{M},\mathcal{N})$ is an $H^*$-submodule of $\bigoplus_{i \in I} Hom_{\mathcal{D}\text{-}Mod}({_{X_i}}{\bf h}\otimes W,~\mathcal{N})$. 

\smallskip
For each $i\in I$, we now prove that $\rho:Hom_{\mathcal{D}\text{-}Mod}({_{X_i}}{\bf h}\otimes W, \mathcal{N}) \longrightarrow Hom_{\mathcal{D}\text{-}Mod}({_{X_i}}{\bf h}\otimes W, \mathcal{N} \otimes H)$, as defined in \eqref{coaction on eta}, gives an $H$-comodule structure on $Hom_{\mathcal{D}\text{-}Mod}({_{X_i}}{\bf h}\otimes W,\mathcal{N})$.  Since $W$ is finite dimensional, we  have
\begin{equation*}
\begin{array}{ll}
Hom_{\mathcal{D}\text{-}Mod}({_{X_i}}{\bf h} \otimes W,~\mathcal{N} \otimes H)&\cong Hom_K(W, Hom_{\mathcal{D}\text{-}Mod}({_{X_i}}{\bf h},~\mathcal{N} \otimes H))\\ & \cong Hom_K(W, \mathcal{N}(X_i) \otimes H)\cong Hom_K(W, \mathcal{N}(X_i))\otimes H \\ 
&\cong Hom_{\mathcal{D}\text{-}Mod}({_{X_i}}{\bf h} \otimes W,~\mathcal{N})\otimes H\\ 
\end{array}
\end{equation*} 
This gives a well defined morphism
\begin{equation}\label{H-comod}
\rho:Hom_{\mathcal{D}\text{-}Mod}({_{X_i}}{\bf h}\otimes W, \mathcal{N}) \longrightarrow Hom_{\mathcal{D}\text{-}Mod}({_{X_i}}{\bf h}\otimes W, \mathcal{N} \otimes H)\cong Hom_{\mathcal{D}\text{-}Mod}({_{X_i}}{\bf h} \otimes W,~\mathcal{N})\otimes H
\end{equation}  We will verify that \eqref{H-comod} gives a right $H$-coaction. For this, we need to show that for any $\zeta \in Hom_{\mathcal{D}\text{-}Mod}({_{X_i}}{\bf h}\otimes W,~\mathcal{N})$, we have $(\rho \otimes \text{id})\rho(\zeta)=(\text{id} \otimes \Delta)\rho(\zeta)$ and $(\text{id} \otimes \varepsilon)\rho(\zeta)=\zeta$. The latter equality is easy to verify. By   \eqref{H-comod}, we know that $\rho(\zeta)= \sum \zeta_0 \otimes \zeta_1 \in Hom_{\mathcal{D}\text{-}Mod}({_{X_i}}{\bf h}\otimes W,~\mathcal{N})\otimes H$. Thus, for any $X \in Ob(\mathcal{D})$ and $u \in {_{X_i}}{\bf h}(X)\otimes W$, we have
\begin{align*}
\left((\rho \otimes \text{id})\rho(\zeta)\right)(X)(u)=\sum \rho(\zeta_0)(X)(u)\otimes \zeta_1 &=\sum (\zeta_0(X)(u_0))_0\otimes S^{-1}(u_1)\big(\zeta_0(X)(u_0)\big)_1\otimes \zeta_1\\
&=\sum (\zeta(X)(u_0))_0\otimes S^{-1}(u_2)(\zeta(X)(u_0))_1\otimes S^{-1}(u_1)(\zeta(X)(u_0))_2\\
&=\sum \zeta_0(X)(u)\otimes \zeta_1 \otimes\zeta_2.
\end{align*}
The third equality above follows by applying $\rho_{\mathcal N(X)} \otimes id_H$ on the equality $\sum \zeta_0(X)(u_0) \otimes \zeta_1=\rho(\zeta)(X)(u_0)$ and the last one is obtained by applying $id_H \otimes \Delta$ on $\sum\zeta_0(X)(u)\otimes \zeta_1=\sum\big(\zeta(X)(u_0)\big)_0 \otimes S^{-1}(u_1)\big(\zeta(X)(u_0)\big)_1$. Thus, we have shown that $Hom_{\mathcal{D}\text{-}Mod}({_{X_i}}{\bf h}\otimes W,~\mathcal{N})$ is a right $H$-comodule.

\smallskip
Moreover, the $H^*$-action on $Hom_{\mathcal{D}\text{-}Mod}({_{X_i}}{\bf h}\otimes W,~\mathcal{N})$ as in \eqref{H^*-action} is  given  precisely by the $H$-coaction as in \eqref{coaction on eta}. Therefore,  $Hom_{\mathcal{D}\text{-}Mod}({_{X_i}}{\bf h}\otimes W,~\mathcal{N})$  is a rational $H^*$-module. Since the category of rational $H^*$-modules contains direct sums (it is equivalent to $Comod\text{-}H$), it follows that $\bigoplus_{i \in I} Hom_{\mathcal{D}\text{-}Mod}({_{X_i}}{\bf h}\otimes W,~\mathcal{N})$ is also a rational $H^*$-module. Being an $H^*$-submodule of $\bigoplus_{i \in I} Hom_{\mathcal{D}\text{-}Mod}({_{X_i}}{\bf h}\otimes W,~\mathcal{N})$, it is now clear that $Hom_{\mathcal{D}\text{-}Mod}(\mathcal{M},\mathcal{N})$ is also a 
rational left $H^*$-module and hence a right $H$-comodule.

\smallskip
It may be verified  that the functor $Hom_{\mathcal{D}\text{-}{Mod}}(\mathcal{M},{-}):{_{\mathcal{D}}}{\mathscr{M}}^H\longrightarrow Comod$-$H$ is right adjoint to the functor $\mathcal{M}\otimes {-}:Comod$-$H\longrightarrow {_{\mathcal{D}}}{\mathscr{M}}^H$ given by 
$N\mapsto\mathcal{M}\otimes N$. Thus, by the uniqueness of adjoints, we have $Hom_{\mathcal{D}\text{-}{Mod}}(\mathcal{M},{-})= HOM_{\mathcal{D}\text{-}{Mod}}(\mathcal{M},{-})$.
\end{proof}

A morphism $\mathcal{N} \longrightarrow \mathcal{N'}$ in ${_{\mathcal{D}}}{\mathscr{M}}^H$ induces a morphism of functors $\mathcal{N}\otimes {(-)}\longrightarrow \mathcal{N'}\otimes {(-)}$ and hence a  morphism $\mathscr{R}_{\mathcal{N'}}\longrightarrow \mathscr{R}_{\mathcal{N}}$ of their respective right adjoints. Thus, for any $\mathcal{L}\in {_{\mathcal{D}}}{\mathscr{M}}^H$, we have a functor $HOM_{\mathcal{D}\text{-}Mod}({-},\mathcal{L}):( {_{\mathcal{D}}}{\mathscr{M}}^H)^{op} \longrightarrow Comod$-$H$ which takes $\mathcal{N}$ to $HOM_{\mathcal{D}\text{-}Mod}(\mathcal{N},\mathcal{L})=\mathscr{R}_{\mathcal{N}}(\mathcal{L})$.

\begin{prop}\label{exact(D,H)}
(1) For any $\mathcal{L}\in {_{\mathcal{D}}}{\mathscr{M}}^H$, the functor $HOM_{\mathcal{D}\text{-}Mod}({-},\mathcal{L}):( {_{\mathcal{D}}}{\mathscr{M}}^H)^{op} \longrightarrow Comod$-$H$ is left exact, i.e., it preserves kernels.

\smallskip
(2)  If $\mathcal{I}$ is injective in ${_{\mathcal{D}}}{\mathscr{M}}^H$, then $HOM_{\mathcal{D}\text{-}Mod}({-},\mathcal{I})$ is exact. 

\smallskip
(3) If $\mathcal I$ is injective  in ${_{\mathcal{D}}}{\mathscr{M}}^H$, then $HOM_{\mathcal{D}\text{-}Mod}({-},\mathcal{I})$ 
takes every short exact sequence in  ${_{\mathcal{D}}}{\mathscr{M}}^H$ to a split short exact sequence in $Comod\text{-}H$. 
\end{prop}

\begin{proof}
(1) Let $\eta: \mathcal{M}\longrightarrow \mathcal{N}$ be a morphism in  ${_{\mathcal{D}}}{\mathscr{M}}^H$ and let
$\mathcal P:=Coker(\eta)$. Then, for any $T\in Comod\text{-}H$, $Coker(\eta\otimes {id}_T:\mathcal M\otimes 
T\longrightarrow\mathcal N\otimes T)=\mathcal P\otimes T$.  From the adjunction in \eqref{iso}, we now have
\begin{equation*}
\begin{array}{l}
Hom_{Comod\text{-}H}\big(T,HOM_{\mathcal{D}\text{-}Mod}(\mathcal{P},~\mathcal{L})\big)\cong Hom_{{_{\mathcal{D}}}{\mathscr{M}}^H}(\mathcal{P}\otimes T,~\mathcal{L})\\
 \cong Ker\Big(Hom_{{_{\mathcal{D}}}{\mathscr{M}}^H}(\mathcal{N}\otimes T,~\mathcal{L})\longrightarrow Hom_{{_{\mathcal{D}}}{\mathscr{M}}^H}(\mathcal{M}\otimes T,~\mathcal{L})\Big)\\
\cong Ker\Big(Hom_{Comod\text{-}H}\big(T,HOM_{\mathcal{D}\text{-}Mod}(\mathcal{N},~\mathcal{L})\big)\longrightarrow Hom_{Comod\text{-}H}\big(T,HOM_{\mathcal{D}\text{-}Mod}(\mathcal{M},~\mathcal{L})\big)\Big)\\
\cong Hom_{Comod\text{-}H}\Big(T, Ker\big(HOM_{\mathcal{D}\text{-}Mod}(\mathcal{N},~\mathcal{L})\longrightarrow HOM_{\mathcal{D}\text{-}Mod}(\mathcal{M},~\mathcal{L})\big)\Big)\\
\end{array}
\end{equation*}
for any $T\in Comod$-$H$. From Yoneda Lemma, it follows that 
\begin{equation*}
HOM_{\mathcal{D}\text{-}Mod}(\mathcal{P},~\mathcal{L})= Ker\big(HOM_{\mathcal{D}\text{-}Mod}(\mathcal{N},~\mathcal{L})\longrightarrow HOM_{\mathcal{D}\text{-}Mod}(\mathcal{M},~\mathcal{L})\big)
\end{equation*}

\smallskip
(2)
 Let $0\longrightarrow \mathcal M
\longrightarrow\mathcal N\longrightarrow\mathcal P\longrightarrow 0$ be a short exact sequence in 
${_{\mathcal{D}}}{\mathscr{M}}^H$. From (1), we already know that 
\begin{equation}\label{lftext}
0\longrightarrow HOM_{\mathcal D\text{-}Mod}(\mathcal P,\mathcal I)\longrightarrow HOM_{\mathcal D\text{-}Mod}(\mathcal N,\mathcal I)\overset{q}{\longrightarrow} HOM_{\mathcal D\text{-}Mod}(\mathcal M,\mathcal I)
\end{equation} is exact.  We need to show that $q$ is an epimorphism. For any $T\in Comod\text{-}H$, we notice that $0\longrightarrow \mathcal M\otimes T
\longrightarrow\mathcal N\otimes T\longrightarrow\mathcal P\otimes T\longrightarrow 0$ is still a short exact sequence in 
${_{\mathcal{D}}}{\mathscr{M}}^H$. If $\mathcal{I}$ is an injective object in ${_{\mathcal{D}}}{\mathscr{M}}^H$, we see that 
\begin{equation*}
0\longrightarrow Hom_{{_{\mathcal{D}}}{\mathscr{M}}^H}(\mathcal{P}\otimes T,~\mathcal{I})\longrightarrow  Hom_{{_{\mathcal{D}}}{\mathscr{M}}^H}(\mathcal{N}\otimes T,~\mathcal{I})\longrightarrow Hom_{{_{\mathcal{D}}}{\mathscr{M}}^H}(\mathcal{M}\otimes T,~\mathcal{I})\longrightarrow 0
\end{equation*}
is an exact sequence of $K$-vector spaces. Using the adjunction in \eqref{iso}, it follows that
\begin{equation}\label{lftext1}
\begin{array}{l}
\begin{CD}
0@>>>  Hom_{Comod\text{-}H}(T,HOM_{\mathcal D\text{-}Mod}(\mathcal P,\mathcal I))
@>>> Hom_{Comod\text{-}H}(T,HOM_{\mathcal D\text{-}Mod}(\mathcal N,\mathcal I))\\
\end{CD}\\
\hspace{2.5in}
\begin{CD}@>Hom(T,q)>>  Hom_{Comod\text{-}H}(T,HOM_{\mathcal D\text{-}Mod}(\mathcal M,\mathcal I))@>>> 0
\\ \end{CD}\\
\end{array}
\end{equation} is short exact in $Vect_K$. By setting $T=HOM_{\mathcal D\text{-}Mod}(\mathcal M,\mathcal I)$ in \eqref{lftext1}, 
we see that there exists a morphism $f:HOM_{\mathcal D\text{-}Mod}(\mathcal M,\mathcal I)\longrightarrow HOM_{\mathcal D\text{-}Mod}(\mathcal N,\mathcal I)$ of $H$-comodules such that $q\circ f$ is the identity on $HOM_{\mathcal D\text{-}Mod}(\mathcal M,\mathcal I)$. This shows that $q:HOM_{\mathcal D\text{-}Mod}(\mathcal N,\mathcal I)\longrightarrow HOM_{\mathcal D\text{-}Mod}(\mathcal M,\mathcal I)$ is an epimorphism.  The result of (3) is clear from the proof of (2). 
\end{proof}

\begin{prop}\label{freeres(D,H)}
Let $\mathcal{D}$ be a left noetherian right co-$H$-category and let $\mathcal{M}\in {_{\mathcal{D}}}{\mathscr{M}}^H$ be finitely generated as an object in $\mathcal{D}$-$Mod$. If $\mathcal{I}$ is an injective object in ${_{\mathcal{D}}}{\mathscr{M}}^H$, then $\textnormal{Ext}_{\mathcal{D}\text{-}Mod}^p(\mathcal{M},\mathcal{I})= 0$ for all $p>0$.
\end{prop}
\begin{proof}
Since $\mathcal{M}\in {_{\mathcal{D}}}{\mathscr{M}}^H$ is finitely generated in $\mathcal{D}$-$Mod$, by Proposition \ref{fg(D,H)}, there exists a finite dimensional $H$-comodule $W_0$ and an epimorphism 
\begin{equation*}
\eta_{0}: \mathcal{P}_0:=\left(\bigoplus_{i=1}^{n_0}{_{X_i}}{\bf{h}} \right)\otimes W_0 \longrightarrow \mathcal{M}
\end{equation*}
in ${_{\mathcal{D}}}{\mathscr{M}}^H$ for finitely many objects $\{X_i\}_{1\leq i \leq n_0}$ in $\mathcal{D}$. Then, $\mathcal{K}:= Ker(\eta_0)$ is a subobject of $\mathcal{P}_0$ in ${_{\mathcal{D}}}{\mathscr{M}}^H$. Since $\mathcal{D}$ is left noetherian, $\mathcal{P}_0$ is a noetherian left $\mathcal{D}$-module (see, for instance, \cite[$\S$ 3]{Mit1}). Thus, the submodule $\mathcal{K}= Ker(\eta_0)$ of $\mathcal{P}_0$ is finitely generated as an object in $\mathcal{D}$-$Mod$. Therefore, we obtain a finite dimensional $H$-comodule $W_1$ and an epimorphism  
\begin{equation*}
\eta_{1}: \mathcal{P}_1:=\left(\bigoplus_{j=1}^{n_1} {_{Y_j}}{\bf h}\right)\otimes W_1 \longrightarrow \mathcal{K}
\end{equation*}
in ${_{\mathcal{D}}}{\mathscr{M}}^H$ for finitely many objects $\{Y_j\}_{1 \leq j \leq n_1}$ in $\mathcal{D}$. Since $W_0$ and $W_1$ are finite dimensional $K$-vector spaces, clearly $\mathcal{P}_0$ and $\mathcal{P}_1$ are free left $\mathcal{D}$-modules. Moreover, $Im(\eta_1) = \mathcal{K}= Ker(\eta_0)$. Continuing in this way, we can construct a free resolution of the module $\mathcal{M}$ in the category $\mathcal{D}$-$Mod$:
\begin{equation*}
\mathcal{P}_*=\hdots \longrightarrow \mathcal{P}_i\longrightarrow\hdots \longrightarrow \mathcal{P}_1\longrightarrow \mathcal{P}_0\longrightarrow \mathcal{M}\longrightarrow 0
\end{equation*}
This gives us
$$ \textnormal{Ext}_{\mathcal{D}\text{-}Mod}^p(\mathcal{M},\mathcal{I})= H^p(Hom_{\mathcal{D}\text{-}Mod}(\mathcal{P}_*,\mathcal{I})),\ \ \ \ \ \ \ \ \ \forall~p>0$$ 
Since $\mathcal{M}$ and $\{\mathcal P_i\}_{i\geq 0}$ are  finitely generated in $\mathcal{D}$-$Mod$, it follows from Proposition \ref{4.12} that   $HOM_{\mathcal{D}\text{-}Mod}(\mathcal{M},\mathcal{I})= Hom_{\mathcal{D}\text{-}Mod}(\mathcal{M},\mathcal{I})$ and $HOM_{\mathcal{D}\text{-}Mod}(\mathcal{P}_i,\mathcal{I})= Hom_{\mathcal{D}\text{-}Mod}(\mathcal{P}_i,\mathcal{I})$. From Proposition \ref{exact(D,H)}, we know that the functor $HOM_{\mathcal D\text{-}Mod}(-,\mathcal I)$ is exact and it  follows that $\textnormal{Ext}_{\mathcal{D}\text{-}Mod}^p(\mathcal{M},\mathcal{I})=H^p(HOM_{\mathcal{D}{\text{-}}Mod}(\mathcal{P}_*,\mathcal{I}))=0$ for all $p>0$. 
\end{proof}

\begin{prop}\label{injres}
Let $\mathcal{D}$ be a left noetherian right co-$H$-category. Let $\mathcal{M},\mathcal{N} \in {_{\mathcal{D}}}{\mathscr{M}}^H$ with $\mathcal{M}$ finitely generated as an object in $\mathcal{D}$-$Mod$. If $\mathcal{E}^*$ is an injective resolution of $\mathcal{N}$ in ${_{\mathcal{D}}}{\mathscr{M}}^H$, then
$$\textnormal{Ext}_{\mathcal{D}\text{-}Mod}^p(\mathcal{M},\mathcal{N})=R^pHOM_{\mathcal{D}\text{-}Mod}(\mathcal{M},\mathcal{N})=H^p\big(Hom_{\mathcal{D}{\text{-}}Mod}(\mathcal{M},\mathcal{E}^*)\big), \quad \forall~p \geq 0.$$
\end{prop}
\begin{proof}
Let $\mathcal{P}_*$ be the free resolution of $\mathcal{M}$ in $\mathcal{D}$-$Mod$ constructed as in the proof of Proposition \ref{freeres(D,H)}. Then, we have
$$\textnormal{Ext}_{\mathcal{D}\text{-}Mod}^p(\mathcal{M},\mathcal{N})=H^p\big(Hom_{\mathcal{D}{\text{-}}Mod}(\mathcal{P}_*,\mathcal{N})\big)=H^p\big(HOM_{\mathcal{D}{\text{-}}Mod}(\mathcal{P}_*,\mathcal{N})\big)$$
where the second equality follows from Proposition \ref{4.12}. Since $HOM_{\mathcal{D}{\text{-}}Mod}(\mathcal{P}_*,\mathcal{N})$ is a complex in $Comod$-$H$ and $Comod$-$H$ is an abelian category, it follows that $H^p\big(HOM_{\mathcal{D}{\text{-}}Mod}(\mathcal{P}_*,\mathcal{N})\big)\in Comod$-$H$. Hence, we may consider the family $\{\textnormal{Ext}_{\mathcal{D}\text{-}Mod}^p(\mathcal{M},-)\}_{p\geq 0}$ as a $\delta-$functor from ${_{\mathcal{D}}}{\mathscr{M}}^H$ to $Comod$-$H$.\smallskip
\newline By Proposition \ref{freeres(D,H)}, $\textnormal{Ext}_{\mathcal{D}\text{-}Mod}^p(\mathcal{M},\mathcal{I})=0$,~$p>0$ for every injective object $\mathcal{I} \in {_{\mathcal{D}}}{\mathscr{M}}^H$. Since ${_{\mathcal{D}}}{\mathscr{M}}^H$ has enough injectives, it follows
  that each $\textnormal{Ext}_{\mathcal{D}\text{-}Mod}^p(\mathcal{M},-):{_{\mathcal{D}}}{\mathscr{M}}^H\longrightarrow Comod$-$H$ is effaceable (see, for instance, \cite[$\S$ III.1]{Hart}). 
  
  \smallskip 
Since ${_{\mathcal{D}}}{\mathscr{M}}^H$ has enough injectives, we can consider the right derived functors
  \begin{equation*} R^p HOM_{\mathcal{D}{\text{-}}Mod}(\mathcal{M},-):{_{\mathcal{D}}}{\mathscr{M}}^H \longrightarrow Comod\text{-}H \qquad p\geq 0
  \end{equation*} For $p=0$, we notice that
$\textnormal{Ext}_{\mathcal{D}\text{-}Mod}^0(\mathcal{M},-)=Hom_{\mathcal{D}\text{-}Mod}(\mathcal{M},-)=HOM_{\mathcal{D}{\text{-}}Mod}(\mathcal{M},-)=R^0HOM_{\mathcal{D}{\text{-}}Mod}(\mathcal{M},-)$ as functors from ${_{\mathcal{D}}}{\mathscr{M}}^H$ to $Comod\text{-}H$. Since each $\textnormal{Ext}_{\mathcal{D}\text{-}Mod}^p(\mathcal{M},-)$ is effaceable for $p>0$, we see that the family $\{\textnormal{Ext}_{\mathcal{D}\text{-}Mod}^p(\mathcal{M},-)\}_{p\geq 0}$ forms a universal $\delta$-functor and it follows from \cite[Corollary III.1.4]{Hart} that
\begin{equation*}\textnormal{Ext}_{\mathcal{D}\text{-}Mod}^p(\mathcal{M},-)=R^pHOM_{\mathcal{D}{\text{-}}Mod}(\mathcal{M},-):{_{\mathcal{D}}}{\mathscr{M}}^H \longrightarrow Comod\text{-}H
\end{equation*} for every $p\geq 0$. Therefore, we have
$$\textnormal{Ext}_{\mathcal{D}\text{-}Mod}^p(\mathcal{M},\mathcal{N})=\left(R^pHOM_{\mathcal{D}{\text{-}}Mod}(\mathcal{M},-)\right)(\mathcal{N})=H^p\big(HOM_{\mathcal{D}{\text{-}}Mod}(\mathcal{M},\mathcal{E}^*)\big)=H^p\big(Hom_{\mathcal{D}{\text{-}}Mod}(\mathcal{M},\mathcal{E}^*)\big)$$ 

\end{proof}

Recall that by Proposition \ref{4.12}, for any $\mathcal{M}\in {_{\mathcal{D}}}{\mathscr{M}}^H$ with $\mathcal{M}$ finitely generated as an object in $\mathcal{D}$-$Mod$, we have $Hom_{\mathcal{D}\text{-}Mod}(\mathcal{M},\mathcal{N})= HOM_{\mathcal{D}\text{-}Mod}(\mathcal{M},\mathcal{N})\in Comod$-$H$.
\begin{theorem}\label{Tf5.17}
Let $\mathcal{D}$ be a left noetherian right co-$H$-category. Let $\mathcal{M}\in {_{\mathcal{D}}}{\mathscr{M}}^H$ with $\mathcal{M}$ finitely generated as an object in $\mathcal{D}$-$Mod$. We consider the functors
\begin{equation*}
\begin{array}{c} \mathscr F=Hom_{\mathcal{D} \text{-}Mod}(\mathcal{M},-): {_{\mathcal{D}}}{\mathscr{M}}^H\longrightarrow {Comod}\text{-}H\qquad \mathcal{N}\mapsto Hom_{\mathcal{D} \text{-}Mod}(\mathcal{M},\mathcal{N})\\
\mathscr G=({-})^{coH}:{Comod}\text{-}H\longrightarrow Vect_K\qquad M\mapsto M^{coH}\\
\end{array}
\end{equation*} Then, we have the following spectral sequence
$$R^{p}(-)^{coH}\big(\textnormal{Ext}^q_{\mathcal{D}\text{-}Mod}(\mathcal{M},\mathcal{N})\big)\Rightarrow \left({R}^{p+q}Hom_{{_{\mathcal{D}}}{\mathscr{M}}^H}(\mathcal{M},-)\right)(\mathcal{N})$$
\end{theorem}
\begin{proof}
By Corollary \ref{H-coH}, we have 
\begin{equation*}(\mathscr{G}\circ\mathscr{F})(\mathcal{N})=Hom_{\mathcal{D}\text{-}Mod}(\mathcal{M},\mathcal{N})^{coH}=HOM_{\mathcal{D}\text{-}Mod}(\mathcal{M},\mathcal{N})^{coH}=Hom_{{_{\mathcal{D}}}{\mathscr{M}}^H}(\mathcal{M},\mathcal{N})
\end{equation*} By definition,
\begin{equation}
R^q\mathscr F(N)=H^q(\mathscr F(\mathcal E^*))=H^q(Hom_{\mathcal D\text{-}Mod}(\mathcal M,\mathcal E^*))
\end{equation} where $\{\mathcal E^*\}$ is an injective resolution of $\mathcal N$ in ${_{\mathcal{D}}}{\mathscr{M}}^H$. Applying Corollary \ref{injres}, we obtain $\textnormal{Ext}_{\mathcal{D}\text{-}Mod}^q(\mathcal{M},\mathcal{N})=R^q\mathscr F(N)$.  For any injective object $\mathcal{I}$ in ${_{\mathcal{D}}}{\mathscr{M}}^H$, we know that $\mathscr{F}(\mathcal{I})=Hom_{\mathcal{D} \text{-}Mod}(\mathcal{M},\mathcal I)=HOM_{\mathcal{D} \text{-}Mod}(\mathcal{M},\mathcal I)$ is injective in $Comod\text{-}H$ by Corollary \ref{4.8}. Since $Comod$-$H$ is a Grothendieck category, it has enough injectives. The result now follows from Grothendieck spectral sequence for composite functors (see \cite{Grothen}).
\end{proof}

\end{document}